\documentclass[11pt,a4paper]{amsart}
\usepackage[utf8]{inputenc}
\usepackage{amsmath}
\usepackage{amsfonts}
\usepackage{amsmath}
\usepackage{amsthm}
\usepackage{amssymb}
\usepackage{graphicx}
\usepackage[shortlabels]{enumitem}
\usepackage{float}
\usepackage{url}
\usepackage{hyperref}		
\usepackage{yfonts}   
\newtheorem{thm}{Theorem}[section]
\newtheorem{theorem}{Theorem}

\newtheorem{lm}[thm]{Lemma}
\newtheorem{pr}[thm]{Proposition}
\theoremstyle{definition}
\newtheorem{df}[thm]{Definition}

\theoremstyle{remark}
\newtheorem{rem}[thm]{Remark}
\newtheorem{clm}[thm]{Claim}

\makeatletter
\newcommand{\xhookdoubleheadrightarrow}[2][]{%
  \lhook\joinrel
  \ext@arrow 0359\rightarrowfill@ {#1}{#2}%
  \mathrel{\mspace{-15mu}}\rightarrow
}
\makeatother
\usepackage[a4paper, left = 2.5cm, right = 2.5cm, top = 2.5cm, bottom = 2.5cm, headsep = 1.2cm]{geometry}
\usepackage[
sorting=nty,
style=alphabetic-verb,
]{biblatex}
\author{Karol Duda}
\title{A computable version of Hall's Harem Theorem and Geometric von Neumann Conjecture}
\date{}
\begin{document}
\begin{abstract} 
We prove a computable version of the Hall Harem Theorem where the matching realizes a unary function with some special properties of its cycles.  
We apply it to non-amenable computable coarse spaces. 
As a result, we obtain a computable version of the geometric von Neumann conjecture.  
\end{abstract}

\maketitle

\section{Introduction} 

Analysis of classical mathematical topics  from the point of view of complexity of various types is one of the major trends of modern mathematical logic. 
Amenability is essentially fruitful from this point of view
(see \cite{kechrisN}, \cite{BK}, \cite{KPT},
\cite{MU}, \cite{CK}, \cite{HPP}, \cite{HKP1} and \cite{HKP2}).
Our research belongs to computable amenability. 
This is a topic where computable versions of fundamentals of amenability are studied; see papers of M. Cavaleri \cite{Cav1, Cav2}, N. Moryakov \cite{mor}, and the author (together with A. Ivanov) \cite{DuI, CPD}. 
In the present paper, we are particularly interested in \textit{the von Neumann conjecture}:  
the statement that each non-amenable group contains a subgroup isomorphic to $\mathbb{F}_2$.  

Let $X$ be a set and let $G$ be a group which acts on $X$ by permutations.
Then the $G$-space $(G,X)$ is called {\em amenable} if there exists a mean $m: \ell^{\infty}(X)\rightarrow \mathbb{R}$ which is invariant under the action of $G$.
In particular, if $X=G$, then the group $G$ is called amenable. 
This definition is due to to J.~von Neumann \cite{Neumann1929}. 
In 1938, Tarski proved his famous Tarski Alternative Theorem \cite{T1}, stating that a $G$-space $X$ is amenable if and only if $X$ does not admit a $G$-paradoxical decomposition. 
Some computable versions of the Tarski Alternative Theorem have been found by the author together with A. Ivanov in \cite{CPD}. 
Roughly speaking \textit{effective paradoxical decompositions} of $G$-spaces of computable transformations were introduced and it was proved that their existence is equivalent to non-amenability.

In \cite{Neumann1929} J.~von Neumann  proved that amenable groups cannot contain a subgroup isomorphic to the two generated free group $\mathbb{F}_2$.
The question whether the converse holds i.e.~whether each non-amenable group contains a subgroup isomorphic to $\mathbb{F}_2$, became popularized as the \textit{von Neumann Conjecture} in the late 1950-s (see e.g. \cite{day}). 
It was answered negatively with the construction of the Tarski monster group by Olshanskii in 1980 \cite{ols}. 
However, in 1999 Whyte \cite{Whyte} proved the geometric version of this conjecture showing existence of $d$-regular forests in non-amenable metric spaces. 
The latter result is now called the \textit{Geometric von Neumann Conjecture}. 
The name corresponds to an easy corollary that so-called the {\em wobbling group} of the space has a subgroup isomorphic to $\mathbb{F}_2$ with a semi-regular action.  
In 2018 Schneider \cite{schndr} generalized Whyte's theorem and the corollary to coarse spaces. 

The main results of our paper are: 
\begin{itemize} 
\item computable versions of Schneider's theorem and the corollary about wobbling groups, 
\item a computable version of so called Hall's harem theorem with controlled sizes of cycles. 
\end{itemize} 
The Hall Harem Theorem is a generalization of Hall's marriage theorem, see \cite[Theorem H.4.2.]{csc}. 
We consider it together with the geometric von Neumann conjecture just because of the well-known fact that Tarski's alternative theorem can be obtained by an application of Hall's harem theorem. 
It turns out that this connection lifts to the level of computability \cite{CPD}, and, on the other hand, to the geometric level.  

\subsection{A computable version of Scheider's theorem}

The theorem of Whyte  mentioned above is as follows. 

\begin{thm}[Geometric Von Neumann Conjecture, \cite{Whyte}]
A uniformly discrete metric space of uniformly bounded geometry is non-amenable if and only if it admits a partition whose pieces are uniformly Lipschitz embbeded copies of the $4$-regular tree.
\end{thm}
Since we concentrate on Schneider's generalization of this theorem \cite{schndr} we skip thorough explanation of the terminology used by Whyte. 
Schneider's theorem \cite{schndr} concerns  
non-amenable coarse spaces. 
We will introduce them to the reader.  
This would give rough understanding of the topic of the Geometric von Neumann Conjecture. 

Let us recall some terminology from  coarse geometry (according to \cite{schndr}). 
For a relation $E\subseteq X\times X$ on a set $X$ and $x\in X,A\subseteq X$, let
$$	
E[x]:=\{y\in X|(x,y)\in E\},
$$
and 
$$
E[A]:=\bigcup\{E[x]|x\in A\}.
$$
Furthermore, we denote by $\Gamma(E)$ the graph associated with the relation $E$, i.e. $\Gamma(E)=(X,E)$. 
\begin{df}
A \textit{coarse space} is a pair $(X,\mathcal{E})$ consisting of a set $X$, and a collection $\mathcal{E}$ of subsets of $X\times X$ (called \textit{entourages}) such that: 
\begin{itemize}
\item the diagonal $\Delta_{X}\in \mathcal{E}$;
\item if $F\subseteq E \in \mathcal{E}$, then also $F\in\mathcal{E}$;
\item if $E,F\in\mathcal{E}$ then $E\cup F, E^{-1}, E\circ F\in\mathcal{E}$.
\end{itemize} 
\end{df}

\begin{df}

A coarse space $(X,\mathcal{E})$ is said to have \textit{bounded geometry} if for each $E\in\mathcal{E}$ and every $x\in X$ the set $E[x]$ is finite. 
\end{df} 
Note that metric spaces are examples of coarse spaces: for a given metric space $(X,\rho )$, a coarse space $(X,\mathcal{E})$ is obtained by setting:
$$
\mathcal{E}:= \{E\subseteq X\times X| \sup \{ \rho (x,y)|(x,y)\in E\}\leq\infty\} . 
$$

\begin{df} 
A coarse space $(X,\mathcal{E})$ of  bounded geometry is called \textit{amenable} if for every  $\theta >1$ and every $E\in\mathcal{E}$ there exists a non-empty finite $F\subseteq X$ such that $|E[F]|\leq \theta|F|.$
\end{df}

\begin{thm} (\cite[Theorem 2.2]{schndr}) 
Let $d\in\mathbb{N}$ and $d\ge 3$. 
A coarse space $(X,\mathcal{E})$ of bounded geometry is not amenable if and only if there is $E\in\mathcal{E}$ such that $\Gamma(E)$ is a $d$-regular forest.
\end{thm} 
To see why this theorem generalizes the theorem of Whyte, note that connected components of $\Gamma(E)$ can be viewed as pieces of the partition from Whyte's theorem. 

In the proof of his result, Schneider takes a symmetric entourage $E$ such that $|E(F)|\geq (d-1)|F|$ for all finite $F\subset X$. Such an entourage exists by the fact 
that when $(X,\mathcal{E})$ is not amenable, then for every finite $d$ there is a symmetric entourage such that $|E(F)|\geq d|F|$ for each finite $F\subseteq X$ (see discussion before Proposition 2.1 in \cite{schndr}). 
Now the bipartite graph $\Gamma:=(X,X,R)$  where $R:=E\setminus \Delta_X$, satisfies Hall's Harem condition for $d-1$: 
\begin{quote} 
for all finite subsets $Y\subset X$ the following inequality holds: 
$|R [Y] |\geq (d-1) |Y|$. 
\end{quote} 
It follows by Hall's harem theorem that there exists an entourage which graph is $d$-regular. 
Moreover, the graph can be transformed into a cycle-free member of $\mathcal{E}$ without losing $d$-regularity. 

In our paper we effectivize this scheme. 
The final goal is formulated as follows. 
(In the formulation the term {\em a highly computable graph} just means a computable locally finite graph where the valency of each vertex can be effectively recognized.) 

\begin{theorem} \label{fg} 
Let $d\geq 3$. 
Let $(\mathbb{N},\mathcal{E})$ be a non-amenable coarse space of a bounded geometry. 
Assume that there exists a highly computable symmetric $E\in\mathcal{E}$ such that for every finite $F\subseteq \mathbb{N}$ we have $|E[F]|\geq (d+2)|F|$. 

Then there exists a computable $E'\in\mathcal{E}$ such that $\Gamma(E')$ is a $d$-regular forest. 
Moreover, there exist an algorithm which for any $m,n\in\mathbb{N}$ recognizes if $m$ and $n$ are in the same connected component of $\Gamma(E')$.
\end{theorem}

To see that this statement is a computable version of Schneider's result, we remind the reader that as we have already mentioned, when $(X,\mathcal{E})$ is not amenable, then for every finite $d$ there is a symmetric entourage such that $|E(F)|\geq d|F|$ for every finite $F\subseteq X$. 

\bigskip 

We now translate Theorem A into a statement about wobbling groups. 
Let us recall that given a coarse space $(X,\mathcal{E})$ its \textit{wobbling group} is 
$$
\mathcal{W}(X,\mathcal{E}):=\{\alpha\in Sym(X)|\{(x,\alpha(x))|x\in X\}\in \mathcal{E}\}.
$$ 
Wobbling groups have attracted growing attention in recent years \cite{Jusch1}, \cite{Jusch2}, 
\cite{schndr}. 
One of motivating points is a possibility of a nice reformulation of  Whyte’s result in terms of semi-regular subgroups. 
Recall that a subgroup $G \leq Sym(X)$ is said to be \textit{semi-regular} if no non-identity element of $G$ has
a fixed point in $X$.
Theorem 6.1 in \cite{Whyte} states that  a uniformly discrete metric space $X$ of uniformly bounded geometry is non-amenable if and only if $\mathsf{F}_2$ is isomorphic to a semi-regular subgroup of $\mathcal{W}(X)$.
Corollary 2.3 of \cite{schndr} states  the same for coarse spaces. 
Note that these formulations look like von Neumann's conjecture! 

The following theorem is a computable version of Corollary 2.3 of \cite{schndr}.

\begin{theorem} 
Let $(\mathbb{N},\mathcal{E})$ be a non-amenable coarse space of a bounded geometry. 
Assume that there exists a highly computable symmetric $E\in\mathcal{E}$ such that $|E[F]|\geq 6|F|$ for all finite $F\subset \mathbb{N}$.

Then there are two computable permutations $\sigma,\pi\in Sym(\mathbb{N})$ such that $\langle\sigma,\pi\rangle$ is a free semi-regular subgroup of $\mathcal{W}(\mathbb{N},\mathcal{E})$.
\end{theorem}

\subsection{Hall's harem theorem} 
Schneider's scheme presented above is based on an application of the Hall harem theorem. 
The latter describes a condition which is equivalent to existence of a perfect $(1,k)$-matching of a bipartite graph; for example, see Section III.3 in \cite{BB}, Theorem 2.4.2 in \cite{1986C1} or Theorem H.4.2 in \cite {csc}. 

In order to obtain a computable version of Schneider's theorem some special form of Hall's harem theorem is required.
Assume that the graph is of the form $(\mathbb{N},\mathbb{N}, E)$, where $E$ denotes the set of edges between natural numbers. 
Any matching clearly realizes a $k$ to $1$ function, say $f: \mathbb{N} \rightarrow \mathbb{N}$. 
The main concern of our paper are properties which can be added to such a function. 
In particular we are especially interested in fast reaching cycles of $f$. 
This will be formalized below as the property of {\em controlled sizes of cycles}. 
As a result, the matching obtained in this way, is realized by a computable subset of $\mathbb{N}^2$ with the additional property that there is an algorithm which decides whether two natural numbers $m$ and $n$ belong to the same connected component of $f$.  

The special form of Hall's harem theorem mentioned in the previous paragraph is the heart of our paper, and the proof of it  (as we evaluate) is 75 per cent of the text. 
In Sections 2 and 3 (a short one) we give an elementary introduction to both graph theoretic and computability theoretic issues. 
In Section 4 we formulate the computable version of Hall's harem theorem, which is our goal. 
There we also describe the strategy of the proof. 
It is worth noting here that the proof is in the standard style of building of computably enumerable sets, as it is usually done in computability theory. 
We worn the reader that the construction is technical and difficult.   
A kind of priority arguments is used in it. 
The proof of it is given in Sections 5 - 8. 
In Section 9 we connect this material with computable amenability. 

\bigskip 

The preprint \cite{HallControlled} contains a  version of Hall's harem theorem with controlled sizes of cycles, where all computability elements are ruled out. 
Although the proof given in it follows Sections 5 - 8 below, but it is slightly easier. 
We should mention that these two versions are incomparable. 
They complement each other and naturally form two parts of a research project. 
The paper \cite{HallControlled} can be used as a helpful companion of the present paper. 
Further details concerning these issues are given in Sections 2 and 4 below.

\begin{rem} 
In the final remark of this introduction we mention other investigations concerning Hall's theorems in the context of computability theory.  
The topic was initiated in papers of Manaster and Rosenstain \cite{MR} and Kierstead \cite{hak}. 
The latter paper was our starting point. 
Hall's theorems have become quite popular in complexity theory. 
For example see \cite{NPmat},  \cite{SMp}.

From the point of view of reverse mathematics Hall's theorems are studied in \cite{CSM}, \cite{MTRM} and \cite{RMMP}. 
\end{rem} 

\section{Preliminaries} 

\subsection{General preliminaries}
To introduce the reader to the subject we recall Hall's Harem theorem. 
We begin with some necessary definitions from graph theory. 
We mostly follow the notation of \cite {csc}. 
A graph $\Gamma=(\mathbf{\Gamma},E)$ is called a \textit{bipartite graph} if the set of vertices $\mathbf{\Gamma}$ is partitioned into sets $U$ and $V$ in such a way, that the set of edges $E$ is a subset of $U\times V$. 
We denote such a bipartite graph by $\Gamma=(U,V,E)$.

From now on we concentrate on bipartite graphs.  
Note that although the definitions below concern this case, they usually have obvious extensions to all ordinary graphs. 

A subgraph of $\Gamma$ is a triple $\Gamma' = (U',V',E')$ with  $U'\subseteq U$, $V'\subseteq V$, $E' \subseteq E$. 
When $\Gamma'$ is such a subgraph but only the sets of its vertices are specified (i.e. for example $\Gamma'=(U',V')$), this means that $E' = E \cap (U'\times V')$. 
Then we say that the subgraph $\Gamma'$ is induced in $\Gamma$ by the set of vertices. 

Let $\Gamma=(U,V,E)$. 
We will say that an edge $(u,v)$ is \textit{incident} to vertices $u$ and $v$. In this case we say that $u$ and $v$ are \textit{adjacent}. 
When two edges $(u,v),(u',v')\in E$ have a common incident vertex we say that $(u,v),(u',v')$ are also \textit{adjacent}. 
A sequence $x_1,x_2,\ldots, x_n$ is called a \textit{path}, if each pair $x_i,x_{i+1}$ is adjacent, $1\leq i< n$.

Below we will denote the set of vertices $\mathbf{\Gamma}$ by the same letter with the graph as a structure, i.e. $\Gamma$. 
Given a vertex $x\in \Gamma$, the \textit{neighborhood} of $x$ is a set 
$$
N _{\Gamma}(x)=\{y\in \Gamma : (x,y)\in E\}.
$$  
For subsets $X\subset U$ and $Y\subset V$, we define the neighborhood 
$N _{\Gamma}(X)$ of $X$ and the neighborhood $N _{\Gamma}(Y)$ of $Y$ by 
\[ 
N _{\Gamma}(X)=\bigcup\limits_{x\in X} N _{\Gamma}(x) \, \, \text{ and } \,  \, N _{\Gamma}(Y)=\bigcup\limits_{y\in Y} N _{\Gamma}(y).
\]  
We drop the subscript $\Gamma$ if it is clear from the context.

\begin{df}\label{connected}
A subset $X$ of $U$ (resp. of $V$) is called \textit{connected} if for all $x,x'\in X$ there exist a path $x=p_0,p_1,\ldots ,p_k=x'$ in $\Gamma$ such that $p_i\in X\cup N_{\Gamma}(X)$ for all $i\leq k$.
\end{df} 
Note here that this definition concerns only bipartite graphs.
This notion belongs to Kierstead, \cite{hak}.
In the case of connected ordinary graphs we take standard definitions.   

For a given vertex $v$, the \textit{star} of $v$ is the subgraph  
$S=(\{v\}\cup N_{\Gamma}(v),E')$ of $\Gamma$, with \[ 
E'=((\{v\}\cup N_{\Gamma}(v))\times (\{v\}\cup N_{\Gamma}(v)))\cap E. 
\]
A $(1,k)$-\textit{fan} is a subset of $E$ consisting of $k$ edges incident to some vertex $u\in U$. 
We say that $u$ is the \textit{root} of the fan, and when $(u,v)$ belongs to the fan we call $v$ a \textit{leaf} of it.

\begin{df}
\begin{itemize} 
\item An $(1,k)$-\textit{matching} from $U$ to $V$ is a collection $M$ of pairwise disjoint $(1,k)$-fans. 
\item The $(1,k)$-matching $M$ is called \textit{left perfect} (resp. \textit{right perfect}) if each vertex from $U$ is a root of a fan from $M$ (resp.  each vertex from $V$ belongs to exactly one fan of $M$).  
\item The $(1,k)$-matching $M$ is called \textit{perfect} if it is both left and right perfect.
\end{itemize} 
\end{df} 
We often view an $(1,k)$-matching as a bipartite graph $M$ where the fan of $u\in U$ is the $M$-\textit{star} of $u$, i.e. the star of $u$ in the subgraph $M$. 
We emphasize that a perfect $(1,k)$-matching from $U$ to $V$ is a set $M\subset E$ satisfying following conditions: 
\begin{enumerate}[label={(\arabic*)}]
\item for each vertex $u \in U$ there exists exactly $k$ vertices $v_1,\ldots, v_k \in V$ such that \\ 
$(u,v_1),\ldots,(u,v_k)\in M$; 
\item for all $v \in V$ there is a unique vertex $u\in U$ such that $(u,v)\in M$.
\end{enumerate} 

Originally Hall's Marriage Theorem (see e.g. \cite[Theorem 2.1.2]{diestel}, \cite[Section III.3]{BB}, \cite[Theorem 2.4.2.]{1986C1}) provides us a condition for existence of left perfect $(1,1)$-matching in a finite bipartite graph. 
We are interested in so called \textit{Hall's Harem theorem}, which is a generalization of Hall's Marriage theorem to the case of perfect $(1,k)$-matchings for the locally finite infinite graphs. 

One says that graph $\Gamma=(V,E)$ is \textit{locally finite}, if for all vertices $x\in \Gamma$, the neighborhood $N(x)$ is finite. Note that if $\Gamma$ is locally finite, then $N(X)$ is finite for any finite $X\subset V$.

\begin{thm}[Hall's Harem theorem]\cite[Theorem H.4.2.]{csc}
Let $\Gamma=(U,V,E)$ be a locally finite graph and let $k\in \mathbb{N},\; k\geq 1$. The following conditions are equivalent: 
\begin{enumerate}[label={(\roman*)}]
\item For all finite subsets $X\subset U$, $Y\subset V$ the following inequalities holds: $|N(X)|\geq k|X|$, $|N(Y)|\geq \frac{1}{k}|Y|$.
\item $\Gamma$ has a perfect $(1,k)$-matching.
\end{enumerate} 
\end{thm} 

The first condition in this formulation is known as \textit{Hall's $k$-harem condition}. 

It is a crucial fact that the theorem  holds for locally finite infinite graphs.
At this point we inform the reader that the statement of the Hall Marriage Theorem does not work for infinite graphs in general (see e.g. \cite[S.2 Ex.6]{diestel}).  Nevertheless, there are versions of this theorem for graphs of any cardinality \cite{ahr}. 

\subsection{Reflections} 
Throughout the paper, $d$ is a natural number greater than $1$.
When $\Gamma=(U,V,E)$ is a bipartite graph, we always assume that $V\subseteq U\subseteq\mathbb{N}$, i.e. $V$ is a subset of the right copy of $U$. 

The following notation substantially simplifies the presentation.
For every vertex $v\in V$ there exist a vertex from $U$ which is a {\em copy} of $v$ (i.e. the same natural number), we denote it by $u_v$.  
If a vertex $u\in U$ has the copy in $V$ then we denote this copy by $v_u$. 

\begin{df} 
The graph $\Gamma=(U,V,E)$ is called $U$-\textit{reflected} if $V$ is a subset of the right copy of $U$ and for every edge $(u,v)\in E$ with $v_{u}\in V$ the edge $(u_{v},v_{u})$ is also in $E$. 
If additionally $V$ is the right copy of $U$, then $\Gamma$ is called a \textit{fully reflected} bipartite graph. 
\end{df}
\noindent 
It is worth noting that when $\Gamma' = (U',V')$ is an induced subgraph of an $U$-reflected  $\Gamma=(U,V,E)$ such that $V'$ is a subset of the right copy of $U'$ then $\Gamma'$ is $U'$-reflected. 

Theorem \ref{hhco} below states that in the case of fully reflected bipartite graphs Hall's harem condition allows us to force some additional properties at the expense of obtaining a perfect $(1,(d-1))$-matching instead of a $(1,d)$-matching.
This is an intermediate statement that avoids computability issues as much as it is possible, the proof of this theorem is available in \cite{HallControlled}. 
We will now give necessary details.

\subsection{Controlled sizes of cycles.}
Let $f$ be a function. 
If for some $i\neq 0$ we have $f^i(u)=u$ then we will say that $u$ is a \textit{periodic point}of $f$. 
For such $u$ and the smallest $i\neq 0$ with $f^i(u)=u$, we say that $\{u,f(u),\ldots, f^{i-1}(u)\}$ is a {\em cycle} of $f$.
Any $(1,(d-1))$ matching can be considered as a $(d-1)$ to $1$ function $f:\mathbb{N}\rightarrow \mathbb{N}$. Moreover, if this matching is perfect, such a function $f$ is total and surjective.
We roughly want to show that given a fully reflected bipartite graph satisfying Hall's $d$-harem condition, there is a perfect $(1,(d-1))$ matching $f:\mathbb{N}\rightarrow\mathbb{N}$, such that for each $u$ there exist $i\geq 0$ such that $f^i(u)$ is a periodic point.

\begin{df}\label{cycles}
Let $f:\mathbb{N}\rightarrow \mathbb{N}$ be a $(d-1)$ to $1$ function. 
We say that $f$ has \textit{controlled sizes of its cycles} if each of the following conditions holds:  
\begin{enumerate}[(i)]
\item $f^2(1)=1$;
\item if $n\geq 2$ and $f^i(n)=n$ then $i\leq n $;
\item if $n\geq 2$ and for all $i\leq n$ we have $f^i(n)\neq n$ then there exist $k\leq 2n$ and $l\leq n$ such that $f^{k+l}(n)=f^k(n)$.
\end{enumerate}
\end{df}  
The following theorem is proved in \cite{HallControlled}. 

\begin{thm}\label{hhco}
Let $\Gamma=(U,V,E)$ be a locally finite bipartite graph such that: 
\begin{itemize}
\item both $U$ and $V$ are identified with $\mathbb{N}\setminus\{0\}$, 
\item $E$ does not contain edges of the form $(u,v_u)$,  
\item  $\Gamma$ is fully reflected,  
\item $\Gamma$ satisfies Hall's $d$-harem condition.
\end{itemize}
Then there exist a perfect $(1,d-1)$-matching of $\Gamma$, which realizes a $(d-1)$ to $1$ function $f:\mathbb{N}\rightarrow \mathbb{N}$ with controlled sizes of its cycles.
\end{thm} 

This is the theorem which effective version is the goal of the present paper. 

\section{Computable versions of Hall's \textit{d}-harem theorem.}\label{compre}

Before we state a computable version of Theorem $\ref{hhco}$ we describe our approach to computability of classical Hall's $d$-harem theorem from \cite{CPD}.

Let us consider the following definition. 
It is standard in computability theory. 

\begin{df} 
A graph $\Gamma$ with the set of vertices $\Gamma$ is \textit{computable} if there exists a bijective function $\nu: \mathbb{N}\rightarrow \Gamma$ such that the set 
$$
R:=\{(i,j): (\nu(i),\nu(j))\in E\} 
$$
is computable. \\ 
A bipartite graph $\Gamma=(U,V,E)$ (with 
$\Gamma = U\, \dot{\cup} \, V$) is \textit{computably bipartite} 
if $\Gamma$ is computable and the set of $\nu$-numbers of $U$ is computable.
\end{df} 

To formulate a computable version of Hall's Harem theorem we also need the following definition of Kierstead \cite{hak}. 

\begin{df}
A locally finite computable graph $\Gamma$ is called \textit{highly computable} if the function  
$g(n)=|N_{\Gamma}(\nu(n))|$, $n\in\mathbb{N}$, is a computable function $g: \mathbb{N}\rightarrow \mathbb{N}$. 
\end{df} 

It is worth noting here that this definition easily implies the existence of an algorithm which for every $m$ and $n$ 
computes the $n$-ball of $\nu (m)$ in $\Gamma$.  

Note that a computably bipartite (resp. highly computable) graph $\Gamma$ with infinite $U$ and $V$ can be always presented in the form $(\mathbb{N},\mathbb{N},E)$ with computable $E$. 
Then induced subgraphs of $\Gamma$ are of the form $(U,V, E\cap (U\times V))$. 
Assuming that $U$ and $V$ are computable subsets of $\mathbb{N}$, the corresponding subgraph admits a computably bipartite (resp. highly computable) presentation.   
Typically we would not insist on finding the latter; it would suffice to know that $U$ is taken from the left copy of $\mathbb{N}$, but $V$ is taken from the right one. 

Manaster and Rosenstein \cite{MR} showed that there is a computably bipartite, highly computable graph that satisfies Hall's condition for a bipartite graph to have a
$(1,1)$-matching, but has no computable $(1,1)$-matching. 

Kierstead \cite{hak} found the natural extension of Hall's condition, called \textit{computable expanding Hall's condition}, which allows us to construct computable $(1,1)$-matchings.

Following \cite{CPD}, in this paper we focus on computable perfect $(1,k)$-matchings. 

\begin{df}\label{cpkm}
Let $\Gamma=(U,V,E)$ be a computably bipartite graph. 
A perfect $(1,k)$-matching $M$ from $U$ to $V$ is called a \textit{computable perfect} $(1,k)$-\textit{matching} if $M$ is a computable set of pairs.
\end{df} 
Observe that computable perfectness exactly means that there is an algorithm which 
\begin{itemize}
\item for each $i\in U$, finds the tuple $(i_1,i_2, \ldots, i_k)$ such that $(i,i_j)\in M$, for all $j=1,2,\ldots, k$;
\item when $i \in V$ it finds $i'\in U$ such that $(i',i)\in M$.
\end{itemize} 

In \cite{CPD} we introduced the following condition, which implies the existence of computable perfect $(1,k)$-matchings in highly computable bipartite graphs.
For $k=1$ it was formulated earlier in \cite{hak}. 

\begin{df} 
A bipartite graph $\Gamma=(U,V,E)$ satisfies 
the \textit{computable expanding Hall's harem condition with respect to $k$} 
(denoted $c.e.H.h.c.(k)$), if and only if there is a total, computable function 
$h: \mathbb{N} \rightarrow \mathbb{N}$ such that:
\begin{itemize}
\item $h(0)=0$
\item for all finite sets $X\subset U$, the inequality $h(n)\leq |X|$ implies $n\leq |N(X)|-k|X|$
\item for all finite sets  $Y\subset V$, the inequality $h(n)\leq |Y|$ implies $n\leq |N(Y)|-\frac{1}{k}|Y|$.
\end{itemize}
\end{df}

In \cite{CPD} we have proven the following.

\begin{thm}\label{HC}
If $\Gamma=(U,V,E)$ is a highly computable bipartite graph satisfying the $c.e.H.h.c.(k)$, then $\Gamma$ has a computable perfect $(1,k)$-matching. 
\end{thm} 

\section{A computable version of Hall's harem Theorem with cycles, preliminaries}
We assume that $\Gamma=(U,V,E)$ is a bipartite graph, such that $V\subseteq U\subseteq\mathbb{N}$, i.e. $V$ is a subset of the right copy of $U$.

\subsection{Main Theorem and the aim of the construction}

In the proof of Theorem \ref{fg} we need a computable perfect $(1,(d-1))$ matching which realizes a function $f:\mathbb{N}\rightarrow\mathbb{N}$, such that the following conditions are satisfied:
\begin{enumerate}[(I)]
\item $\{n\, | \, n \text{ is a periodic point}\}$ is computable;
\item $\{(n,m)\, | \, n \text{ is a periodic point }, m \text{ is in the cycle of } n\}$ is computable;
\item the set $\{(n,m)\, | \, f^m(n) \text{ is a periodic point}\}$ is computable and its first coordinates cover $\mathbb{N}\setminus \{ 0\}$. 
\end{enumerate}
Moreover, the matching has to be constructed by a uniform algorithm.

In fact it is easy to see that a computable $(d-1)$ to $1$ function with controlled sizes of its cycles (Definition \ref{cycles})
 satisfies conditions (I)--(III).

The following theorem is a kind of amalgamation of Theorem \ref{hhco} and Theorem \ref{HC}. 

\begin{thm}[Main Theorem]\label{ehhc}
Let $\Gamma=(U,V,E)$ be a highly computable bipartite graph, such that: 
\begin{itemize}
\item both $U$ and $V$ are identified with $\mathbb{N}\setminus\{0\}$
\item $E$ does not contain edges of the form $(u,v_u)$
\item  $\Gamma$ is fully reflected,  
\item $\Gamma$ satisfies the $c.e.H.h.c.(d)$.
\end{itemize}
Then there exist a computable perfect $(1,d-1)$-matching of $\Gamma$, which realizes a computable $(d-1)$ to $1$ function $f:\mathbb{N}\rightarrow \mathbb{N}$ with controlled sizes of its cycles.
\end{thm} 

The proof of the theorem will be finished in  Section \ref{pMT}. 
The next sections focus on a construction used in it. 

\subsection{Notation used in the construction} 
Since this construction is highly technical, 
we start with a list of the notation. 
We do not insist on a thorough inspection of it. 
In the beginning a hasty view will suffice. 
\begin{itemize}

\item $M$ is a perfect matching that we construct.

\item $M_{n-1}$ is a set of $(1,d-1)$-fans added to $M$ at the end of the $n$-th step. Thus $M=\bigcup\limits_{n=1}^{\infty}M_{n-1}$.

\item $\Gamma^{(0)}=(U^{(0)},V^{(0)},E^{(0)})$ is the original graph $\Gamma$.

\item $U^{(n)}:=U^{(n-1)}\setminus \{u\in U^{(n-1)}: \exists v\in V^{(n-1)}, (u,v)\in M_{n-1}\}$.

\item $V^{(n)}:=V^{(n-1)}\setminus \{v\in V^{(n-1)}: \exists u\in U^{(n-1)}, (u,v)\in M_{n-1}\}$.

\item $\Gamma^{(n)}=(U^{(n)},V^{(n)})$. 
We will see that $\Gamma^{(n)}$ is $U^{(n)}$-reflected.

\item After the $n$-th step we obtain decompositions 
$U^{(n)}= U^{(n)\star}\, \dot{\cup} \, U^{(n)\perp}$ and  
$V^{(n)}= V^{(n)\star}\, \dot{\cup} \, V^{(n)\perp}$,  
where we say that $U^{(n)\perp}$ consists of elements from $U^{(n)}$ which might spoil Hall's $d$-harem condition for $\Gamma^{(n)}$. 

\item Put $U^{(0)\perp} =\emptyset$ and $V^{(0)\perp}=\emptyset$ 
(since $\Gamma^{(0)}$ is $\Gamma$, i.e. it satisfies Hall's $d$-harem condition). 

\item $\Gamma^{(n)\perp}$ is a graph with the sets of vertices $(U^{(n)\perp},V^{(n)\perp})$ and the set of edges corresponding to $(1,d-1)$-fans with roots denoted by $u^{\perp}_i\in U^{(n)\perp}$. 

\item When $U^{(n)\perp}\setminus U^{(n-1)\perp}$ is not empty,  $U^{(n)\perp}\setminus U^{(n-1)\perp}=\{u^{\perp}_{n-1}\}$ and $V^{(n)\perp}\setminus V^{(n-1)\perp}$ consists of leaves $\{v^{\perp}_{n-1,i}:1\leq i\leq d-1\}$. 

\item  $\Gamma^{(n)\star}:=(U^{(n)}\setminus U^{(n)\perp}, V^{(n)}\setminus V^{(n)\perp})$. 
We will see that $\Gamma^{(n)\star}$ satisfies Hall's $d$-harem condition.

\item During the construction of $M_{n}$ we will define fans $M^j_{n}$, $j\leq n+1$. 
The graph $M_{n}$ is the union of them. 

\item $M^j_{n}$ is a fan consisting of edges denoted by $\{(u^{j}_n,v^j_{n,i}): i\leq d-1\}$.

\item $u_n$ is a starting vertex of the $n$-th step, it is also denoted by $u^0_n$.
%%%%%%%%%%%
%\item if $v_{u_n^j}=v^{\perp}_{k,l}$ for some $k,l$, 
%then  $u_n^{j+1}:=u^{\perp}_{k}$. 
%Otherwise $u^{j+1}_n:=u_{v^j_{n,1}}$.

\item For any subgraph $\Gamma'$ of $\Gamma^{(n)}$ we denote by $\Gamma'(-u^{0}_n,\ldots, -u^{j}_n)$ the graph obtained from $\Gamma'$ by removal of the $(1,d-1)$-fans of $M_n$ with roots $u^{0}_n,\ldots, u^{j}_n$. 

\item For any subgraph $\Gamma'=(U',V')$ of $\Gamma^{(n)}$ and any $u_j^{\perp}\in U^{(n)\perp}$ we denote by $\Gamma'(+u_j^{\perp})$ the graph induced in $\Gamma^{(n)}$ by the sets of vertices $U'\cup\{u_j^{\perp}\}$ and $V'\cup\{v^{\perp}_{j,i}: 1\leq i\leq d-1\}$. 

\item For any subgraph $\Gamma'=(U',V')$ of $\Gamma^{(n)}$ and any vertex $v\in V^{(n)}$ we denote by $\Gamma'(+v)$ (resp. $\Gamma'(-v)$) the graph induced in $\Gamma^{(n)}$ by the sets of vertices $U'$ and $V'\cup\{v\}$ (resp. $V'\setminus \{ v\}$). 

\item We denote by $h(x)$ a witness of $c.e.H.h.c.(d)$ for $\Gamma$. 
At the end of each step we produce $\widehat{h}(x)$ (possibly with some index) which is a witness of $c.e.H.h.c.(d)$ for $\Gamma^{(n)\star}$, defined with the help of $h$. 
See Lemma \ref{2nd} for details. 

\item The elements adjacent to $u^{j+1}_n$ in the matching $\mathfrak{M}^2_{n}$ are denoted by $\dot{v}^{j+1}_{n,1},\ldots, \dot{v}^{j+1}_{n,d}$. 
The element $\dot{v}^{j+1}_{n,1}$ is a candidate for $v^{j+1}_{n,1}$. 

\item $\mathcal{B}^{(n)}(u_n)$ (resp. $\mathcal{S}^{(n)}(u_n)$) is the ball (resp. sphere) in $\Gamma^{(n)}$  with the center $u_n$ and the radius $\max\{4h(3d(n+1))+3,5\}$.

\item $\mathcal{B}^{(n)}(u^{l+1}_n)$ (resp. $\mathcal{S}^{(n)}(u^{l+1}_n)$) is the ball (resp. sphere) in $\Gamma^{(n)}(-u^{0}_n,\ldots, -u^{l}_n)$  with the center $u^{l+1}_n$ and the radius $\max\{4h(3d(n+1))+3,5\}$.

\item $\mathfrak{M}^1_n$ denotes a finite $(1,d)$-matching in the bipartite graph 
$\mathcal{B}^{(n)\star}(u_n)=\mathcal{B}^{(n)}(u_n)\cap \Gamma^{(n)\star}$. 
It is obtained in such a way, that it satisfies the conditions of the perfect $(1,d)$-matchings for all vertices that are at the distance less than $\max\{4h(3d(n+1))+3,5\}$ from $u_n$.

\item $\mathfrak{M}^2_n$ denotes a finite $(1,d)$-matching in the bipartite graph 
$\mathcal{B}^{(n)\star}(u^{l+1}_n)=\mathcal{B}^{(n)}(u^{l+1}_n)\cap \Gamma^{(n)\star}$. 
It is obtained in such a way, that it satisfies the conditions of the perfect $(1,d)$-matchings for all vertices that are at the distance less than $\max\{4h(3d(n+1))+3,5\}$ from $u^{l+1}_n$.

\item The fan $(u^{\perp}_{n}, \{ v^{\perp}_{n,j} \, |1\leq j\leq d-1 \})$ usually appears as a part of a fan of the matching $\mathfrak{M}^2_{n}$. 
We warn the reader that it is possible that $u^{\perp}_{n}$ does not exist.   

\item We assume that all of these elements (i.e. $u_n, v^j_{n,i}, u^j_n , v^{\perp}_{j,i}, \ldots$ ) are natural numbers and are ordered according to the standard ordering of the natural numbers.
\end{itemize}

The main difference with the non-effective version of the construction given in \cite{HallControlled} is that instead of $(1,d)$-matchings $\mathfrak{M}^1_n$, $\mathfrak{M}^2_n$ introduced in the list above (when we produce $(1,d-1)$-fans to add to the constructed objects), the circumstances of \cite{HallControlled} allow us to  use perfect $(1,d)$-matchings of infinite graphs. 
So, the construction in that paper is simpler. 
In the present paper we need some additional effort (mainly in Sections \ref{lemmastechnic} and \ref{lemmas}) in order to show that under an appropriate choice of the radius, the finite matchings $\mathfrak{M}^1_n$ and $\mathfrak{M}^2_n$ work without spoiling $c.e.H.h.c.(d)$ in the resulting graphs.

Before describing the construction in detail, we show the single step of the algorithm in the picture attached below, see Figure 1. 
We inform the reader that detailed illustrations of all parts of the step will be given in the appropriate places.

\begin{figure}[ht]
  \centering
    \includegraphics[width=1\textwidth]{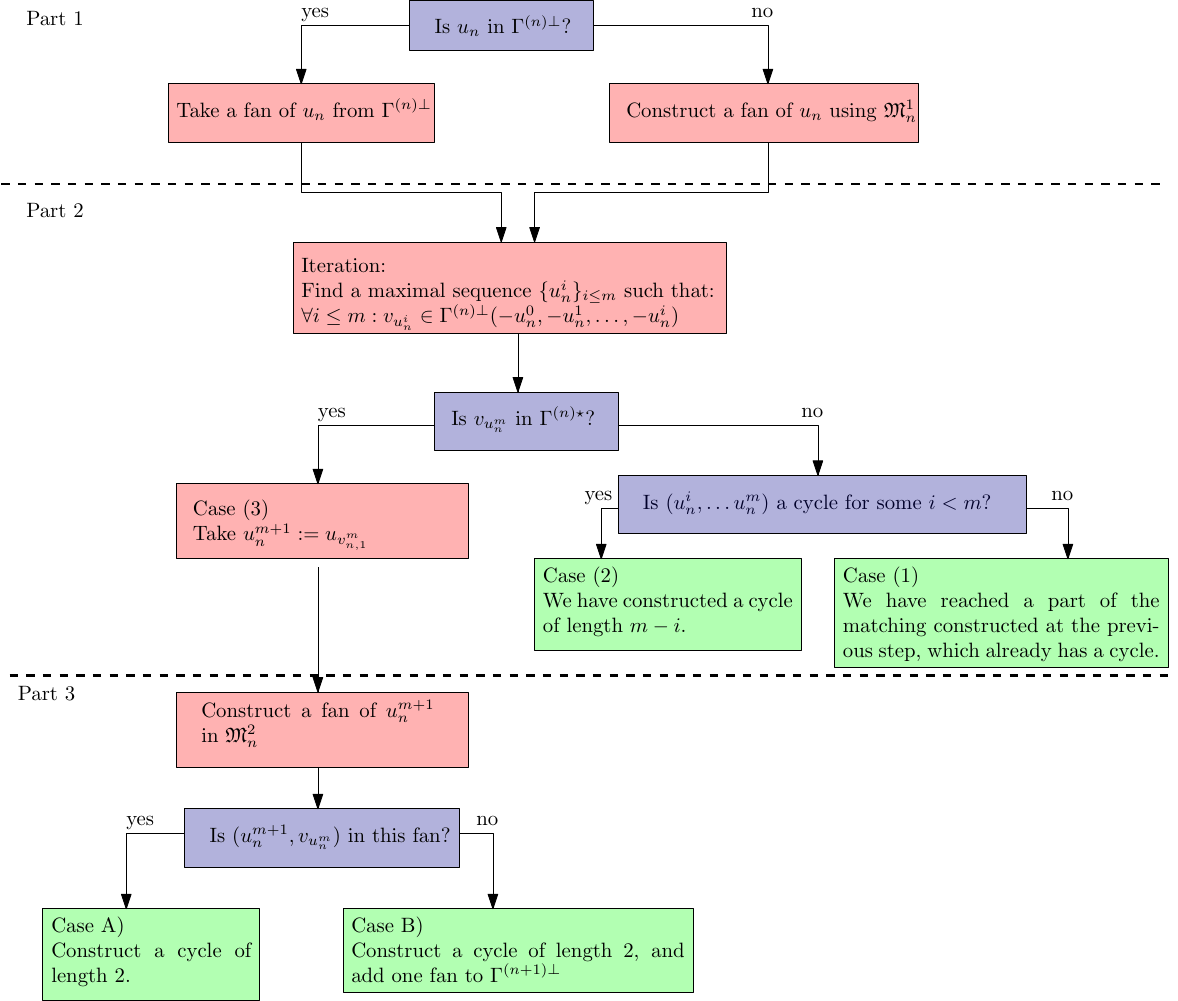}
    	\caption{Single step.}
\end{figure}

\section{Construction}\label{const}  

From now on in this section we fix a highly computable bipartite graph $\Gamma=(U,V,E)$ that satisfies $c.e.H.h.c.(d)$ and the following conditions:
\begin{itemize}
\item both $U$ and $V$ are identified with $\mathbb{N}\setminus\{0\}$;
\item $\Gamma$ is fully reflected;   
\item $E$ does not contain edges of the form $(v,u_v)$. 
\end{itemize} 
Before the construction of the proof of Theorem \ref{ehhc} we formulate an  easy proposition which will be often used below without any additional comments.  
It gives an algorithm for finite matchings $\mathfrak{M}^1_n$ and $\mathfrak{M}^2_n$ in $\mathcal{B}^{(n)\star}(u^{0}_n)$ and in $\mathcal{B}^{(n)\star}(u^{j}_n)$ respectively (see the notation of Section 4.2).

\begin{pr}\label{finmat}
Let $\Gamma'=(U',V',E')$ be a subgraph of $\Gamma = (U,V,E)$ which is also a highly computable bipartite graph that satisfies $c.e.H.h.c.(d)$. 
Then there is a uniform algorithm as follows. 

Having an input $u^* \in U'$ and an odd radius $r$ let $\mathcal B(u^*)$ and $\mathcal S(u^*)$ denote the corresponding ball and the sphere of the radius $r$ in $\Gamma$. 
Put $\mathcal{B}'(u^*) = \mathcal{B}(u^*)\cap \Gamma'$ 
\footnote{note that $\mathcal{B}'(u^*)$ does not have to be a ball in $\Gamma'$ of the center $u^*$}.

Then the algorithm finds a $(1,d)$-matching in $\mathcal{B}'(u^* )$, say $M_{u^*}$, which satisfies the conditions of perfect $(1,d)$-matchings for all $u\in U'\cap\mathcal{B}(u^*)$ and $v\in V'\cap(\mathcal{B}(u^*)\setminus \mathcal{S}(u^*))$. 
\end{pr}

From now on, such a matching  $M_{u^*}$ will be called a \textit{perfect $(1,d)$-matching in the ball} $\mathcal{B}'(u^* )$.

\begin{proof}
All elements of $V'\cap(\mathcal{B}(u^*)\setminus\mathcal{S}(u^*))$ have the same neighborhoods in $\Gamma'$ and $\mathcal{B}'(u^*)$.
By Theorem \ref{HC} there exists a perfect $(1,d)$-matching in $\Gamma'$. 
Thus there exists a matching in $\mathcal{B}'(u^*)$ which satisfies the conditions of a perfect $(1,d)$-matchings for all $u\in U'\cap\mathcal{B}(u^*)$ and $v\in V'\cap(\mathcal{B}(u^*)\setminus \mathcal{S}(u^*))$. 
Since $\Gamma'$ is highly computable and $\mathcal{B}'(u^*)$ is finite, we find this matching in a computable way. 

\end{proof} 
We now describe an inductive construction that is the heart of the proof of Theorem \ref{ehhc}.  
Every step consists of three parts. 
Some of them are finished by claims stating that certain graphs satisfy Hall's $d$-harem condition. 
These statements support further stages of the construction. 
We start with a detailed description of the first step of the construction.

\subsection{Step 1, part 1} 
We take $u_0$, the first element of the set $U$, i.e. $u_0=1$.
Let $\mathcal{B}^{(0)}(u_0)$ be the ball centered at $u_0$ and of the radius $\max\{4h(3d)+3,5\}$. 
Apply Proposition \ref{finmat} and compute a perfect $(1,d)$-matching $\mathfrak{M}^1_{0}$, in $\mathcal{B}^{(0)}(u_0)$. 
Let $v^0_{0,1},\ldots, v^0_{0,d}$ be elements of $V$, ordered by its numbers, such that $(u_0,v^{0}_{0,i})\in \mathfrak{M}^1_{0}$ for all $i\leq d$. 
Define the fan $M^0_0$ as the set of edges $(u_0,v_{0,i})$ for $i\leq d-1$. 
Let 
\[ 
\Gamma^{(0)}(-u_0):=(U\setminus\{u_0\},V\setminus\{v^0_{0,1},\ldots, v^0_{0,d-1}\}). 
\] 

\begin{figure}[ht]
  \centering
    \includegraphics[width=0.7\textwidth]{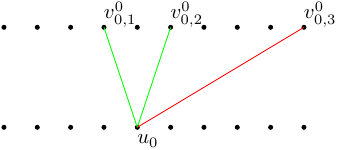}
	\caption{The first part of step 1, the $\mathfrak{M}^1_{0}$-fan of $u_0$ where $M^0_0$ is green.}
\end{figure}

\begin{clm}\label{c1p1ef}
The graph $\Gamma^{(0)}(-u_0)$ satisfies $c.e.H.h.c.(d)$.
\end{clm}
This claim follows from Lemma \ref{1st} which will be proved of in Section 7.

\subsection{Remarks before part 2}
Before the second part of step 1 let us discuss our local goals. 
Let $f_0$ be the partial function which  corresponds to $M_0$ and let $\Gamma^{(1)}$ be the graph obtained  after step 1, i.e. the part of $\Gamma$ after removal of $M_0$. 
We want to force that:
\begin{enumerate}
\item for all $n \in \mathsf{Dom}(f_0)$  there exists $i$ such that $f_0^i(n)$ is a periodic point, and
\item $\Gamma^{(1)}$ is $U^{(1)}$-reflected.
\end{enumerate}

It is clear that (1)-(2) are satisfied if we add the edge $(u_{v^0_{0,1}},v_{u_0})$ to the matching. 
This means that $f(v_{u_0}) = u_{v^0_{0,1}}$, i.e.  $f_0^2(u_0)= f_0 (f_0 (v_{u_0})) = u_0$.

\subsection{Step 1, part 2}

Denote $u^{1}_0:=u_{v^0_{0,1}}$. 
Note that $(u^1_0,v_{u_0})\in\Gamma^{(0)}(-u_0)$, because $(u_0,v^0_{0,1})$ is in $\Gamma^{(0)}$ and the latter one is $U^{(0)}$-reflected. 

\begin{rem}
For arbitrary $n$ part 2 of step $n$ depends on the subgraph $\Gamma^{(n-1)\perp}$. Since $\Gamma^{(0)\perp}$ is empty, at step 0 this part is reduced just to the choice of the vertex $u^{1}_0$. 
\end{rem}
In part 3 we will add the edge $(u^{1}_0, v_{u_0})$, to $M_0^1$. 

\subsection{Step 1, part 3} 

Let $\mathcal{B}^{(0)}(u^{1}_0)$ be the ball in $\Gamma^{(0)}(-u_0)$  with the center $u^{1}_0$ and the radius $\max\{4h(3d)+3,5\}$.
Since $\Gamma^{(0)}(-u_0)$ satisfies $c.e.H.h.c.(d)$, we can apply Proposition \ref{finmat} for $\Gamma^{(0)}(-u_0)$ (viewed as $\Gamma=\Gamma'$ in its formulation) and compute $\mathfrak{M}^2_{0}$, a perfect $(1,d)$-matching in $\mathcal{B}^{(0)}(u^1_0)$.
We remind the reader that $\dot{v}^1_{0,1},\ldots, \dot{v}^1_{1,d}$ are elements of $V$, ordered by their numbers, such that $(u^1_0,\dot{v}^1_{0,i})\in \mathfrak{M}^2_{0}$ for all $i\leq d$.

Since $v_{u_0}$ is the least number in $V$, there are two possible cases:

\begin{enumerate}

\item $v_{u_0}=\dot{v}^1_{0,1}$. We set $v^1_{0,i}:=\dot{v}^1_{0,i}$, $1\le i\le d-1$ (i.e. $v^1_{0,1}=v_{u_0}$). 

\item $v_{u_0}\neq \dot{v}^1_{0,1}$. 
Then find the fan $(u,v_i)\in \mathfrak{M}^2_{0}$, $1\leq i\leq d$, such that $v_{u_0}=v_1$ (assuming that the ordering of $v_i$ corresponds to their indexes). 
We set:
\begin{itemize}
\item $v^1_{0,1}:=v_{u_0}$ and $v^1_{0,i}:=\dot{v}^1_{0,i-1}$ for $2 < i \le d-1$,   
\end{itemize} 
and define a candidate for $\Gamma^{(1)\perp}$: 
\begin{itemize} 
\item $\dot{u}^{\perp}_0:=u$;
\item $\dot{v}^{\perp}_{0,i-1}:=v_i$ for $1<i\le d$.
\end{itemize}
\end{enumerate}

In either case define the fan $M^1_0$ as the set of edges $(u^1_0,v^1_{0,i})$ for $1\leq i\leq d-1$. 

\begin{figure}[ht]
  \centering
    \includegraphics[width=0.65\textwidth]{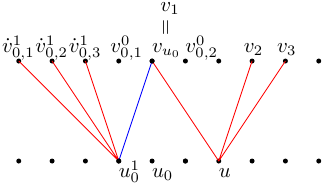}
    \caption{Step 1, part 3. $\mathfrak{M}^2_{0}$ is red and $v_{u_0}$ is matched with $u$. 
We want the edge $(u^{1}_0, v_{u_0})$ to be in $M_0$. 
Force the situation from Figure 4.}
\end{figure}

\begin{figure}[ht]
  \centering
    \includegraphics[width=0.65\textwidth]{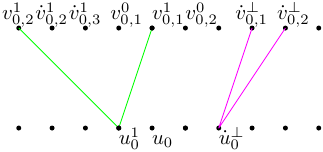}
    \caption{$M^1_{0}$ is green. It is possible that the purple fan consisting of edges $(\dot{u}^{\perp}_0,\dot{v}^{\perp}_{0,1}),(\dot{u}^{\perp}_0,\dot{v}^{\perp}_{0,2})$ will be added to $\Gamma^{(1)\perp}$.}
\end{figure}

Put $M_0=M^0_0\cup M^1_0$. 
We obtain $\Gamma^{(1)}$ by removal of $M_0$ from $\Gamma$. Since $u_0, v_{u_0}, v^0_{0,1}$ and $u^1_0=u_{v^0_{0,1}}$ have been removed, $\Gamma^{(1)}$ is $U^{(1)}$-reflected. 
It might turn out that it does not satisfy Hall's $d$-harem condition. 

Let $\Gamma'_0=(U^{(1)}\setminus \{\dot{u}_0^{\perp}\},V^{(1)}\setminus \{\dot{v}^{\perp}_{0,1},...,\dot{v}^{\perp}_{0,d-1}\})$. 
The statement below follows from Lemma \ref{2nd}.

\begin{clm}\label{c1p2ef}
At least one of $\Gamma'_{0}$ or $\Gamma^{(1)}$  
satisfies $c.e.H.h.c.(d)$.
\end{clm}

\subsection{The output of the first step}

If $\Gamma^{(1)}$ satisfies $c.e.H.h.c.(d)$, set
$\Gamma^{(1)\star}:=\Gamma^{(1)}$,  $U^{(1)\perp}=\emptyset$ and  $V^{(1)\perp}=\emptyset$.
If $\Gamma^{(1)}$ does not satisfy $c.e.H.h.c.(d)$, set
$\Gamma^{(1)\star}:=\Gamma'_{0}$ and 
\begin{itemize}
\item $u^{\perp}_0:=\dot{u}^{\perp}_0$;
\item $v^{\perp}_{0,i}:=\dot{v}^{\perp}_{0,i}$;
\item $U^{(1)\perp}=\{u^{\perp}_0\}$;
\item $V^{(1)\perp}=\{v^{\perp}_{0,i}: 1 \leq i\leq d-1\}.$
\end{itemize}

\subsection{The situation before step n+1}
We have constructed graphs $\Gamma^{(n)}$ and $\Gamma^{(n)\star}$, where $\Gamma^{(n)}$ is $U^{(n)}$-reflected and $\Gamma^{(n)\star}$ satisfies $c.e.H.h.c.(d)$. 
Since $|U^{(n)\perp}\setminus U^{(n-1)\perp}|\leq 1$, there are at most $n$ roots $u^{\perp}_i$ of fans in $U^{(n)\perp}$ (see Section 2.4 for the corresponding definition).

\subsection{Step n+1, part 1} 

Take $u_n$, the first element of the set $U^{(n)}$. 
In order to define $M^0_n$ we have two possible cases. 

\begin{enumerate} 
\item There is $j$ such that $u_n=u_j^{\perp}\in U^{(n)\perp}$. Then we set $M^0_n$ to consist of all edges of the form $(u^{\perp}_j,v^{\perp}_{j,i})$ and remove the fan with the root $u_j^{\perp}$ from $\Gamma^{(n)\perp}$. We redefine $U^{(n)\perp}$ and $V^{(n)\perp}$ accordingly (in particular $u_j^{\perp}$ is removed from $U^{(n)\perp}$). 

\item If $u_n\not\in U^{(n)\perp}$, then take the ball $\mathcal{B}^{(n)}(u_n)$ in $\Gamma^{(n)}$ with the center $u_n$ and the radius $\max\{4h(d(3n+1))+3,5\}$.  
Applying Proposition \ref{finmat} compute $\mathfrak{M}^1_n$, a perfect $(1,d)$-matching in the ball $\mathcal{B}^{(n)\star}(u_n):=\mathcal{B}^{(n)}(u_n)\cap\Gamma^{(n)\star}$. 
Let $v^0_{n,1},\ldots, v^0_{n,d}$ be elements of $V^{(n)\star}$, ordered by its numbers, such that $(u_n,v^{0}_{n,i})\in \mathfrak{M}^1_{n}$ for all $i\leq d$.
Verify whether there is $j$ with $(u_{v^0_{n,j}},v_{u_n})\in \Gamma^{(n)\perp}$. 
By the definition of $\Gamma^{(n)\perp}$ it can happen for at most one $j$.\footnote{Note here that it can also happen that $v_{u_n}$ is not even in $\Gamma^{(n)}$.}  
If there is such $j$ then it can be  equal to $d$ or not. 
In the first case, set $M^0_n$ to consist of edges $(u_n,v^0_{n,i})\in \mathfrak{M}^1_n$ for $i\neq d-1$. 
In all other possibilities, we set $M^0_n$ to consist of edges $(u_n,v^0_{n,i})\in \mathfrak{M}^1_n$ for $i\leq d-1$.
\end{enumerate}
\begin{rem} 
Note that as a result in the case of existence of $j$ as in (2) the edge $(u_n, v^0_{n,j})$ is included into $M^0_n$. 
Although the final part of the algorithm of (2) can be presented easier, i.e. as in Figure \ref{fig:part1},   
we do not want to miss this point. 
\end{rem}

\begin{figure}[ht]
  \centering
    \includegraphics[width=1\textwidth]{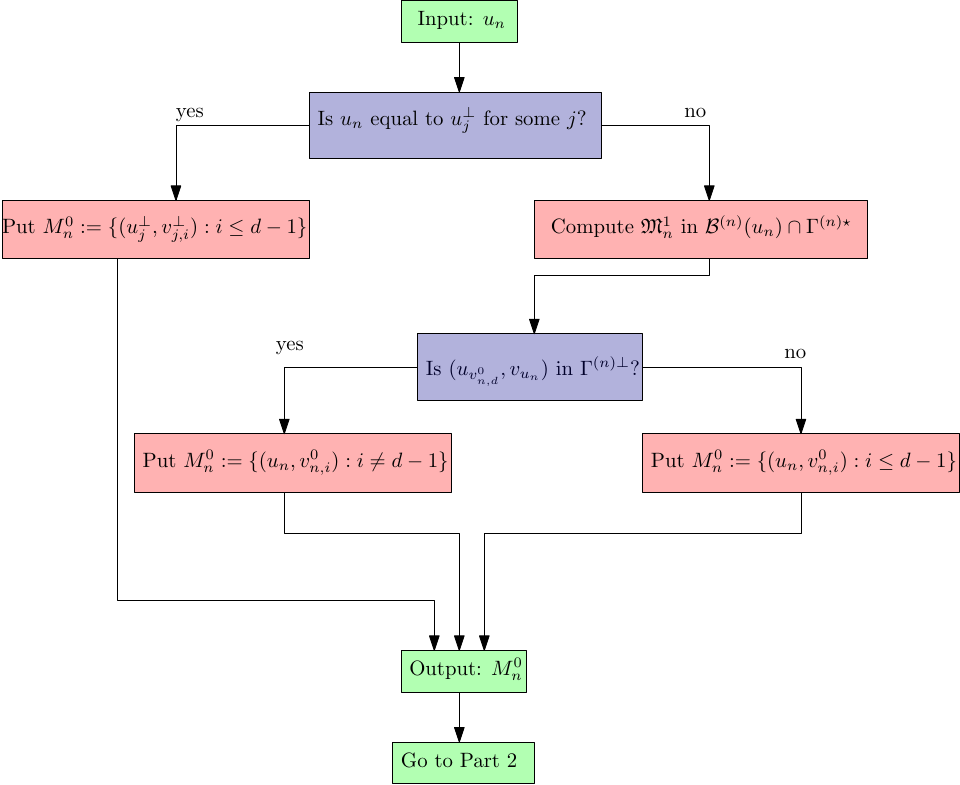}
\caption{Detailed version of the algorithm used at part 1 of step $n+1$.}\label{fig:part1}
\end{figure}

\pagebreak 

\begin{clm}\label{csnp1ef}
Let $\Gamma^{(n)\star}(-u_n)$ be $\Gamma^{(n)}(-u_n)\cap \Gamma^{(n)\star}$.
Then one of the following conditions holds:
\begin{itemize}
\item $\Gamma^{(n)\star}(-u_n)$ satisfies $c.e.H.h.c.(d)$;
\item there is $u_j^{\perp}\in U^{(n)\perp}$ such that the graph $\Gamma^{(n)\star}(-u_n,+u_j^{\perp})$ satisfies $c.e.H.h.c.(d)$.
\end{itemize}
There is an algorithm which verifies these conditions and finds $u^{\perp}_j$ in the second one. 
\end{clm} 
In case (1),  $\Gamma^{(n)\star}(-u_n) =  \Gamma^{(n)\star}$ and the claim is obvious. In case (2) the claim follows from Lemma \ref{1st} below.

\subsection{The output of part 1}  
This is $M^0_n$. 
We also update our graphs in the following situation.  
If $\Gamma^{(n)\star}(-u_n)$ does not satisfy $c.e.H.h.c.(d)$, let $u_j^{\perp}$ be an element from $U^{(n)\perp}$ realizing the second possibility of Claim \ref{csnp1ef}. 
We remove the fan of $u^{\perp}_j$ with its leaves 
from $(U^{(n)\perp},V^{(n)\perp})$ and then we put it into $\Gamma^{(n)\star}$.
Thus the latter graph (and $U^{(n)\perp},V^{(n)\perp}$) are updated.
It is clear that now the redefined $\Gamma^{(n)\star}(-u_n)$ satisfies $c.e.H.h.c.(d)$.

\subsection{Remarks before part 2} \label{Rbp2} 
Before the rest of step $n+1$, we describe the goals which we want to achieve after the step: 

\begin{enumerate}
\item  the partial function $f_n$ corresponding to $\bigcup\limits_{i=0}^{n} M_i$ has controlled sizes of its cycles, and
\item the graph $\Gamma^{(n+1)}$ obtained at the end of the step, is $U^{(n+1)}$-reflected, and the corresponding graph $\Gamma^{(n+1)\star}$ satisfies Hall's $d$-harem condition.
\end{enumerate}
In order to achieve the first condition we will organize one of the the following properties:
\begin{enumerate}[(i)]
\item there is a sequence of vertices $u^0_n,u^1_n, u^2_n,\ldots, u^j_n$, $1\leq j\leq n$, such that every edge $(u^i_n,v_{u^{i-1}_n})$ belongs to $M_n$, and for some $0\leq \ell \leq j-1$ the sequence $(u^{\ell}_n, u^{\ell +1}_n,\ldots, u^j_n)$ is a cycle;
\item there is a sequence of vertices $u^0_n, u^1_n, u^2_n,\ldots, u^j_n$, $j\leq n$, such that each edge $(u^i_n,v_{u^{i-1}_n})$ belongs to $M_n$, and $v_{u^j_n}$ is already adjacent to some edge from $\bigcup\limits_{i=0}^{n-1} M_i$;
\end{enumerate}

\subsection{Step n+1, part 2} \label{nplus1p2}

We begin by checking whether $v_{u_n}$ belongs to $\Gamma^{(n)\perp}$.  
If $v_{u_n}\in \Gamma^{(n)\perp}$ then we denote $u_n$ by $u^0_n$ and begin the following process of choosing the consecutive vertices $u^i_n$. 
%\vspace{1cm} 
\begin{quotation}
\textit{First step of iteration.} 
Assume that for some $j_0,i$ we have $v_{u_n}=v^{\perp}_{j_0,i}\in V^{(n)\perp}$.
Then we set 

$u^1_n := u^{\perp}_{j_0}$, $v^1_{n,k}:=v^{\perp}_{j_0,k}$, $k\le d-1$, 

$M^1_n:=\{(u^1_n,v^1_{n,k}): 1 \le k \le d-1 
\}$. \\ 
and check whether $v_{u^1_n}\in \Gamma^{(n)\perp}(-u^0_n)$. \\ 
If it is so we repeat the iteration for $v_{u^1_n}$. 
Note that $v_{u^1_n} \in \Gamma^{(n)\perp}(-u^0_n,-u^1_n )$
then. 
\end{quotation} 
%\vspace{1cm}
\begin{quotation}
\textit{Single step of iteration.} 
We verify if 
$v_{u^m_n}\in \Gamma^{(n)\perp}(-u^0_n,-u^1_n,\ldots, -u^{m}_n)$. 
If it is so then for some $j_{m},i$ we have $v_{u^m_n}=v^{\perp}_{j_{m},i}\in V^{(n)\perp}$. 
Define 

$u^{m+1}_n := u^{\perp}_{j_m}$,  $v^{m+1}_{n,k}:=v^{\perp}_{j_m,k}$, $1\le k \le d-1$, 

$M^m_n:=\{(u^{m+1}_n,v^{m+1}_{n,k}): 1 \le k \le d-1 
\}$, 

and we repeat the iteration for $v_{u^{m+1}_n}$.
This ends the single iteration step.
\end{quotation}
%\vspace{1cm}

Since $|U^{(n)\perp}|\leq n$, the procedure ends after at most $n$ iterations. 
Therefore one of the following cases is realized for some $l\leq n$: 

\begin{enumerate}
\item $v_{u^l_n}\notin \Gamma^{(n)}$;
\item $v_{u^l_n}\in \Gamma^{(n)\perp}$, but $v_{u^l_n}\notin \Gamma^{(n)}(-u^0_n,-u^1_n,\ldots, -u^{l}_n)$ (this case is impossible for $l=0$);
\item $v_{u^l_n} \in \Gamma^{(n)\star}$ (and obviously $v_{u^l_n}\in \Gamma^{(n)}(-u^0_n,-u^1_n,\ldots, -u^{l}_n)$).
\end{enumerate}

In case (1), $v_{u^l_n}$ was already added to $M$ at preceding steps and condition (ii) of Remarks at point  \ref{Rbp2} is satisfied.
We finish part 2 of step $n+1$ without a new cycle. 

In case (2), the last iteration closes  the cycle $(u^k_n,\ldots, u^l_n)$. 
The length of this cycle is not greater than $l+1$.

In each of these two cases we skip part 3 of the step and go to the output of step $n+1$, see \ref{otpt}.
We set $M_n=\bigcup\limits_{k=0}^{l} M^k_n$ and obtain the graph $\Gamma^{(n+1)}$ from $\Gamma^{(n)}$ by removal of $M_n$-fans. 
Since for every $u \in U^{(n)}\setminus U^{(n+1)}$ the element $v_u$ is either removed as well or was not in $V^{(n)}$ from the beginning, then the graph is $U^{(n+1)}$-reflected. 

Case (3) is the most complicated one; part 3 of the step will be entirely dedicated to it. 
In this case, we finish part 2 of the step by taking a new term $u^{l+1}_n := u_{v^l_{n,1}}$. 
In part 3 we will force the cycle $(u^{l}_n,u^{l+1}_n)$ of length $2$.

\bigskip 

Figure \ref{fig:part2} illustrates the algorithm of the second part of the step.

\begin{figure}[ht]
  \centering
    \includegraphics[width=1\textwidth]{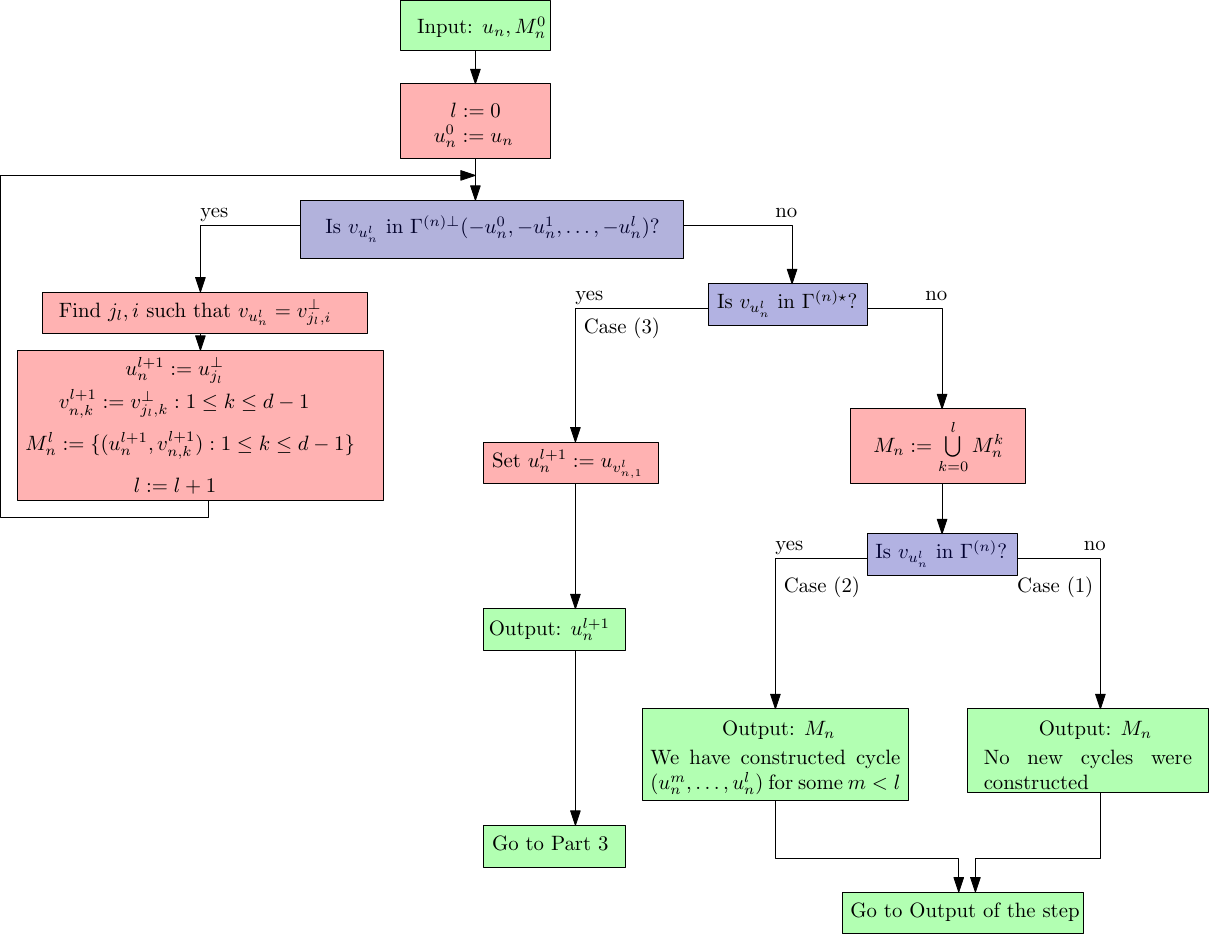}
    \caption{Detailed version of the algorithm of part 2 of step $n+1$.}\label{fig:part2}
\end{figure}

\pagebreak 

\subsection{Example}  \label{cycllemma}

Before we move to the third part of the step, we give an example of a cycle obtained by the algorithm in case (2) of this part of the step. 

The following pictures show how a cycle of length $3$ arises in this procedure when the graph satisfies Hall's $3$-harem condition.

\begin{figure}[ht]
  \centering
    \includegraphics[width=0.8\textwidth]{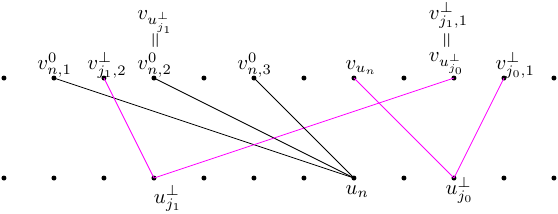}
    \caption{$\Gamma^{(n)\star}$ is black, $\Gamma^{(n)\perp}$ is purple.}
\end{figure}

\begin{figure}[ht]
  \centering
    \includegraphics[width=0.8\textwidth]{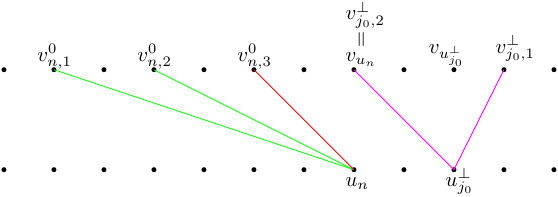}
    \caption{$\mathfrak{M}^n_1$ is red and green, $M_n^0$ is green. We have $v_{u_n}=v^{\perp}_{j_0,2}$.}
\end{figure}

\begin{figure}[ht]
  \centering
    \includegraphics[width=0.8\textwidth]{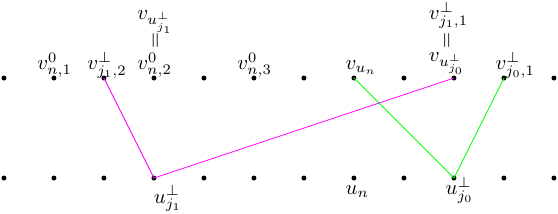}
    \caption{$M_n^1$ is green. We have $v_{u^{\perp}_{j_1}}=v^{\perp}_{j_2,2}$. Moreover $v_{u_{j_1}^{\perp}}=v^0_{n,2}$.}
\end{figure}

To demonstrate the cycle in a single picture, we take 4 copies of $\mathbb{N}$ as rows. 
The first and the second ones correspond to  $V$ and $U$ after part 1 of the step. 
The second and the third rows correspond to the sets $V$ and $U$ after the first iteration step. 
The third and the fourth rows correspond to $U$ and $V$ after the second step of the iteration.

\begin{figure}[ht]
  \centering
    \includegraphics[width=0.8\textwidth]{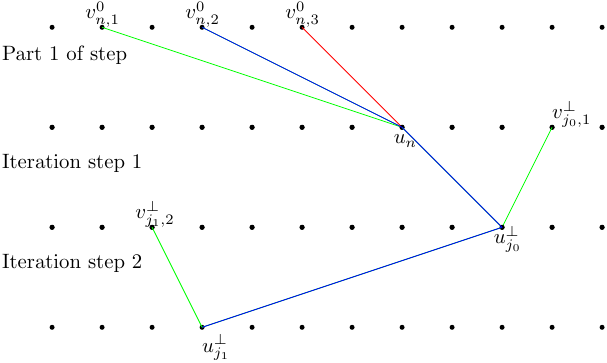}
    \caption{The final $M_n$ with the obtained cycle of length $3$ marked in blue. The red edge is the edge from $\mathfrak{M}^n_1$ that was not added to $M_n^0$.}
\end{figure}

%\vspace{0.5cm}
\subsection{Step n+1, part 3}
Let  $\mathcal{B}^{(n)}(u^{l+1}_n)$ 
be the ball in $\Gamma^{(n)}(-u_n,-u^1_n,\ldots,-u^l_n)$ with the center $u^{l+1}_n$ and the radius $\max\{4h(3d(n+1))+3,5\}$. 
Let $\mathcal{B}^{(n)\star}(u^{l+1}_n):=\mathcal{B}^{(n)}(u^{l+1}_n)\cap\Gamma^{(n)\star}$.

Case (3) of part 2 guarantees that the edge $(u^l_n,v^l_{n,1})$ is in $\Gamma^{(n)}$ and $v_{u^l_n} \in \Gamma^{(n)}(-u_n,-u^1_n,\ldots,-u^l_n)$.  
Thus applying $U^{(n)}$-reflectedness of $\Gamma^{(n)}$ we see $(u^{l+1}_n,v_{u^l_n})\in\Gamma^{(n)}(-u_n,-u^1_n,\ldots,-u^l_n)$. 
Observe that since in part 2 the elements $u^{\perp}_{j_1},\ldots, u^{\perp}_{j_l}$ were chosen in $U^{(n)\perp}$, 
$$
\Gamma^{(n)}(-u_n,-u^1_n,\ldots,-u^l_n) \cap\Gamma^{(n)\star}=\Gamma^{(n)}(-u_n)\cap\Gamma^{(n)\star}=\Gamma^{(n)\star}(-u_n).
$$  

By part 1 of this step the graph $\Gamma^{(n)\star}(-u_n)$ satisfies $c.e.H.h.c.(d)$, thus we can apply Proposition \ref{finmat} for $\Gamma^{(n)\star}(-u_n)$ (viewed as $\Gamma'$ in the formulation, with $\Gamma^{(n)}(-u_n,-u^1_n,\ldots,-u^l_n)$ viewed as $\Gamma$) and compute $\mathfrak{M}^2_n$, a perfect $(1,d)$-matching in the ball $\mathcal{B}^{(n)\star}(u^{l+1}_n)$.

\begin{quotation} 
We now check whether there is $j$ with $u^{l+1}_n=u^{\perp}_j$. 
If it exists, we set 
$\dot{v}^{l+1}_{n,i}:=v^{\perp}_{j,i}$ for $1\leq i\leq d-1$.
When it does not exist, $\dot{v}^{l+1}_{n,1},\ldots,  \dot{v}^{l+1}_{n,d}$ will 
denote the elements adjacent to $u^{l+1}_n$ under $\mathfrak{M}^2_n$. 

There are two cases:

\begin{enumerate}[label={\Alph*)}]
\item $(u^{l+1}_n,v_{u^l_n})\in \mathfrak{M}^2_n$, i.e. $v_{u^l_n}=\dot{v}^{l+1}_{n,k}$ for some $1\leq k \leq d$. In this case $u^{l+1}_n$ cannot be $u^{\perp}_j$ for any $j$.\label{caseA1}

\item $(u^{l+1}_n,v_{u^l_n})\notin \mathfrak{M}^2_n$, i.e. there exists some $u\in \Gamma^{(n)\star}(-u_n)$, such that $v_{u^l_n}=v_k$ for some $1\leq k \leq d$, where $v_1,\ldots, v_d$ denote the elements adjacent to $u$ under $\mathfrak{M}^2_n$. In this case it is possible that $u^{l+1}_n$ coincides with some $u^{\perp}_j$. \label{caseB1}
\end{enumerate}

In either case we produce a cycle of length 2 by including the pair $(u^{l+1}_n,v_{u^{l}_n})$ into $M^{l+1}_n$. 
In fact we include it into $M^{l+1}_n$ together with a fan with the root $u^{l+1}_n$ and $(d-2)$ leaves taken among $\dot{v}^{l+1}_{n,i}$. 

To be precise, we organize it as follows. 

In case \ref{caseA1} 
if $k=d$, we set $v^{l+1}_{n,i-1}:=\dot{v}^{l+1}_{n,i}$ for $2\leq i\leq d$.
If $k\neq d$ we set $v^{l+1}_{n,i}:=\dot{v}^{l+1}_{n,i}$ for $1\leq i\leq d-1$. 
The set $M^{l+1}_n$ consists of edges $(u^{l+1}_n,v^{l+1}_{n,i})$ for $1\leq i\leq d-1$. 
The procedure is finished.

In case \ref{caseB1} $v^{l+1}_{n,1}=v_{u^{l}_n}$ and $v^{l+1}_{n,i+1}=\dot{v}^{l+1}_{n,i}$, $1 \le i \le d-2$.
The set $M^{l+1}_n$ consists of edges $(u^{l+1}_n,v^{l+1}_{n,i})$ for $1\leq i\leq d-1$. 
We define $\dot{u}_n^{\perp}:=u$ and rename the remaining $d-1$ vertices $v_j$ to $\dot{v}^{\perp}_{n,i}$.
The procedure is finished.

\vspace{0.5cm}
\end{quotation}

Figure \ref{fig:part3} illustrates part 3 of the algorithm. 

\begin{figure}[ht]  \centering
    \includegraphics[width=1\textwidth]{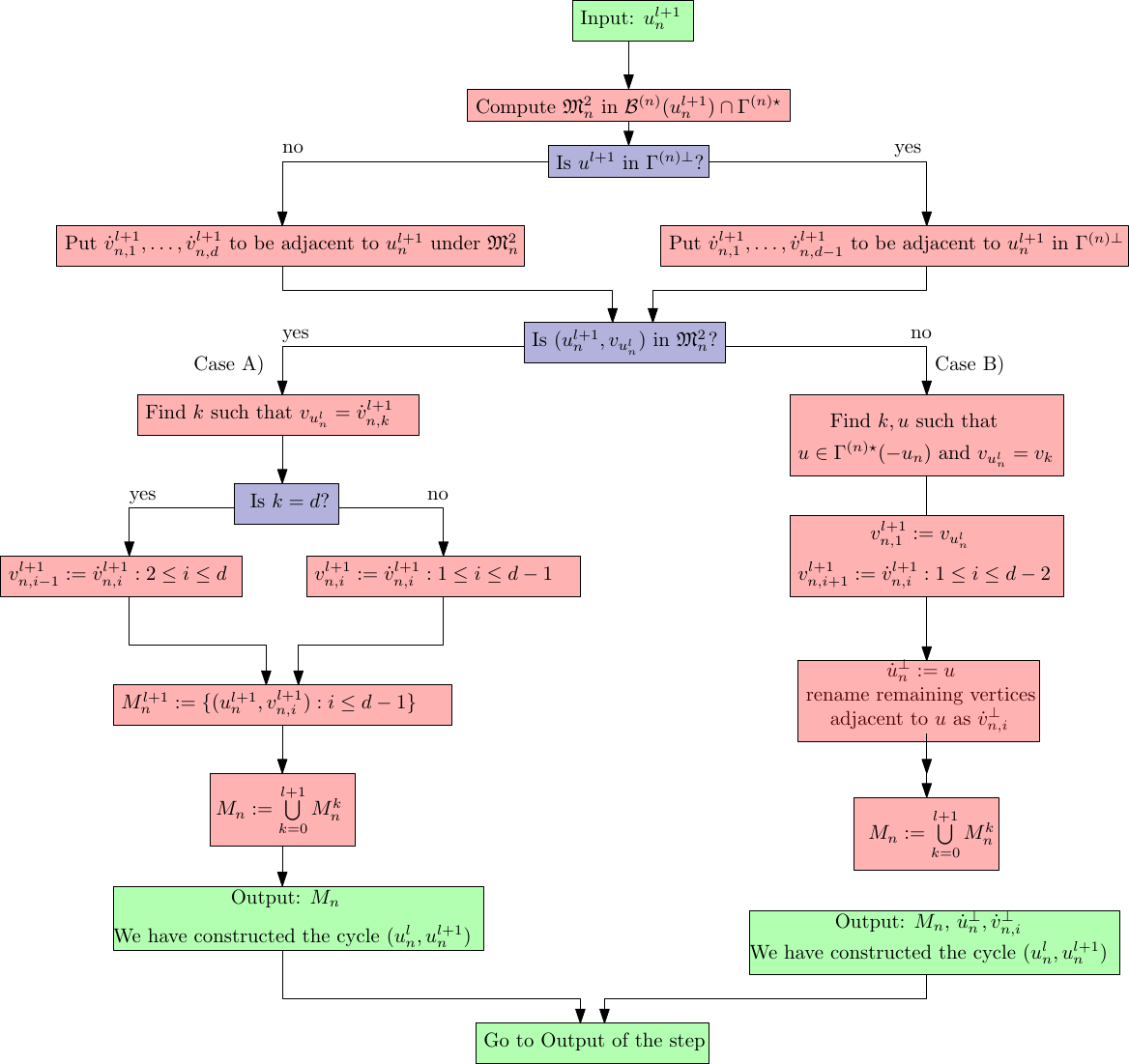}
    \caption{Detailed version of the algorithm used at the part 3 of the step $n+1$.}\label{fig:part3}

\end{figure}

\pagebreak 

Let $M_n=\bigcup\limits_{k=0}^{l+1} M^k_n$.
We obtain the graph $\Gamma^{(n+1)}$ from $\Gamma^{(n)}$ by removal of $M_n$-fans.
Since for each $u \in U^{(n)}\setminus U^{(n+1)}$ the element $v_u$ is also removed, then the graph is $U^{(n+1)}$-reflected. 

\subsection{Output of the step} \label{otpt} 

We have already defined $M_n$ and $\Gamma^{(n+1)}$. 
It remains to produce graphs $\Gamma^{(n+1)\star},\Gamma^{(n+1)\perp}$. 
For this purpose, we define auxiliary graphs 
$$ 
\mathfrak{T}=(\mathfrak{U},\mathfrak{V}) \mbox{ and }\dot{\Gamma}^{(n)\perp} = (\dot{U}^{(n)\perp},\dot{V}^{(n)\perp}).
$$
In cases (1), (2) and (3)A) of parts 2 and 3 we set $\mathfrak{U}:=U^{(n+1)}\setminus U^{(n)\perp}$,  $\mathfrak{V}:=V^{(n+1)}\setminus V^{(n)\perp}$, and $\dot{\Gamma}^{(n)\perp}:=\Gamma^{(n)\perp}\cap \Gamma^{(n+1)}$.
Note here that in cases (1) and (2), in part 2 of the step, only elements of $\Gamma^{(n)\perp}$ were added to $M_n$. 
Therefore, in these cases, $\mathfrak{T}$ coincides with $\Gamma^{(n)\star}(-u_n)$ and satisfies Hall's $d$-harem condition.

In case (3)B) we set 
\begin{itemize} 
\item $\mathfrak{U}:=U^{(n+1)}\setminus (U^{(n)\perp}\cup \{\dot{u}_n^{\perp}\})$, 
\item $\mathfrak{V}:=V^{(n+1)}\setminus (V^{(n)\perp} \cup \{\dot{v}_{n,i}^{\perp}: 1\leq i\leq d-1\})$,  
\end{itemize} 
unless $u^{l+1}_n$ coincides with some $u^{\perp}_j$. 
In the latter case, $\dot{v}^{l+1}_{n,d-1}$ is the only vertex left from the fan of $u^{\perp}_j$ in $\Gamma^{(n)\perp}$. 
\footnote{Furthermore, the vertex $\dot{v}^{l+1}_{n,d}$ does not exist, i.e the only vertex from the fan of $u^{l+1}_n$ that is outside of $M^{l+1}_n$ is $\dot{v}^{l+1}_{n,d-1}$.}
Consequently, we define 

$\bullet \, \mathfrak{V}:= (V^{(n+1)}\cup \{ \dot{v}^{l+1}_{n,d-1}\})\setminus (V^{(n)\perp} \cup \{\dot{v}_{n,i}^{\perp}: 1\leq i\leq d-1\})$.

We define $\dot{U}^{(n)\perp}$ and  $\dot{V}^{(n)\perp}$ accordingly: 
\begin{itemize} 
\item $\dot{U}^{(n)\perp}:=(U^{(n)\perp}\cup \{\dot{u}_n^{\perp}\})\cap U^{(n+1)}$, 
\item  $\dot{V}^{(n)\perp}:=(V^{(n)\perp}\cup \{\dot{v}_{n,k}^{\perp}:1\leq k \leq d-1\})\cap V^{(n+1)}$,  
\end{itemize} 
The following claim follows from Lemma \ref{2nd} below.

\begin{clm}\label{cs1p2ef}
At least one of the following holds:
\begin{itemize}
\item $\mathfrak{T}$ satisfies  $c.e.H.h.c.(d)$;
\item there exists some vertex $u_j^{\perp}\in \dot{U}^{(n)\perp}$ such that the graph $\mathfrak{T}(+u_j^{\perp})$ satisfies $c.e.H.h.c.(d)$.
\item there exist vertices $u_i^{\perp},u_j^{\perp}\in \dot{U}^{(n)\perp}$ such that the graph $\mathfrak{T}(+u_i^{\perp},+u_j^{\perp})$ satisfies  $c.e.H.h.c.(d)$.
\end{itemize}
\end{clm}

Now, depending on the output of the claim above we define graphs $\Gamma^{(n+1)\star},\Gamma^{(n+1)\perp}$.

In case (3)B) (i.e. $\dot{u}_n^{\perp}$ exists) if  $u_i^{\perp}\neq \dot{u}_n^{\perp}\neq u_j^{\perp}$ (for any output of the claim), we set $u_n^{\perp}:=\dot{u}_n^{\perp}$ and $v_{n,k}^{\perp}:= \dot{v}_{n,k}^{\perp}$ for $1\leq k\leq d-1$. 
Otherwise, or in the remaining cases $u_n^{\perp},v_{n,k}^{\perp}$ are not defined.

In the first case of the claim we set $\Gamma^{(n+1)\star}:=\mathfrak{T}$ and \begin{itemize} 
\item $U^{(n+1)\perp}:= (U^{(n)\perp}\cup\{u_n^{\perp}\})\cap U^{(n+1)},$
\item $V^{(n+1)\perp}:=((V^{(n)\perp}\setminus\{\dot{v}^{l+1}_{n,d-1}\})\cup \{v_{n,k}^{\perp}: 1\leq k\leq d-1\})\cap V^{(n+1)}.$
\end{itemize} 
In the second case we set $\Gamma^{(n+1)\star}:=\mathfrak{T}(+u_j^{\perp})$ and 
\begin{itemize} 
\item $U^{(n+1)\perp}:= ((U^{(n)\perp}\setminus\{u_j^{\perp}\})\cup \{u_n^{\perp}\})\cap U^{(n+1)},$
\item $V^{(n+1)\perp}:=(((V^{(n)\perp}\setminus\{\dot{v}^{l+1}_{n,d-1}\})\setminus \{v_{j,k}^{\perp}: 1\leq k\leq d-1\})\cup \{v_{n,k}^{\perp}: 1\leq k\leq d-1\})\cap V^{(n+1)}.$
\end{itemize} 
In the third case we set $\Gamma^{(n+1)\star}:=\mathfrak{T}(+u_i^{\perp},+u_j^{\perp})$ and 
\begin{itemize} 
\item $U^{(n+1)\perp}:= ((U^{(n)\perp}\setminus \{u_i^{\perp},u_j^{\perp}\})\cup\{u_n^{\perp}\})\cap U^{(n+1)},$
\item $V^{(n+1)\perp}:=(((V^{(n)\perp}\setminus\{\dot{v}^{l+1}_{n,d-1}\})\setminus \{v_{i,k}^{\perp},v_{j,k}^{\perp}: 1\leq k\leq d-1\})\cup \{v_{n,k}^{\perp}: 1\leq k\leq d-1\})\cap V^{(n+1)}.$ 
\end{itemize}

\section{General Lemmas}\label{lemmastechnic}

The construction presented in Section 5 is supported by the lemmas of Section 7. 
In order to prove them, we need some additional observations of slightly more general character. 
This is the purpose of this section. 
We warn the reader that the notation used here does not coincide with that of Section 5. 
\begin{itemize} 
\item Throughout this section $\Gamma=(U,V,E)$ denotes a bipartite graph  \\ 
and $\Gamma^{\perp}=(U^{\perp},V^{\perp},E^{\perp})$ denotes its subgraph. \\ 
The graph $\Gamma^{\star}=(U^{\star},V^{\star},E^{\star})$ is an induced subgraph of $\Gamma$ such that 
$$
U^{\star}\cap U^{\perp}=\emptyset = V^{\star}\cap V^{\perp}.  
$$ 
\item Below we always assume that $d$ is a natural number greater than $1$. 
\end{itemize}  
The following situation will arise several times in our arguments in Section \ref{lemmas}. 
Let $X$ be a subset of $V$ such that  
$$|N_{\Gamma}(X)|\geq \frac{1}{d}|X|,$$
but
$$
|N_{\Gamma}(X) \setminus U^{\perp}|<\frac{1}{d}|X|.
$$
Thus we can conclude that $N_{\Gamma}(X)\cap U^{\perp}\neq \emptyset$.

The following lemma describes typical circumstances which lead to this situation.

\begin{lm}\label{neighsize}
Let $\Gamma=(U,V,E)$ be $U$-reflected and  
let $\Gamma^{\star}$ be a subgraph of $\Gamma$ induced by the sets of vertices $U^{\star}=U\setminus U^{\perp}$, $V^{\star}=V\setminus V^{\perp}$.
Assume that $\Gamma^{\star}$ satisfies Hall's $d$-harem condition, and assume that for each $Y\subset U^{\perp}$ we have 
$|N_{\Gamma}(Y) \cap V^{\perp}|\geq (d-1)|Y|$.

Then for any $X\subseteq V$ we have $|N_{\Gamma}(X)|\geq (d-1)|X|$.
In particular $|N_{\Gamma}(X)|\geq \frac{1}{d}|X|$
\end{lm}
\begin{proof}
Let $U_X:=\{u\in U: v_u\in X\}$.
Since $\Gamma$ is $U$-reflected, $|U_X|=|X|$.
Consider the sets 
$$U_X^{\star}:=U_X\setminus U^{\perp}$$
and 
$$U_X^{\perp}:=U_X\cap U^{\perp}.$$  
Using $U$-reflectedness of $\Gamma$ again, we see  
$$
|N_{\Gamma}(X)|\geq |N_{\Gamma^{\star}}(U_X^{\star})|+|N_{\Gamma}(U^{\perp}_X)\cap V^{\perp}|.$$
Since Hall's $d$-harem condition is satisfied for any subset of $U^{\star}$, 
$$
|N_{\Gamma^{\star}}(U_X^{\star})|\geq d|U_X^{\star}|.
$$
Moreover we have $|N_{\Gamma}(U_X^{\perp})\cap V^{\perp}|\geq (d-1)|U_X^{\perp}|$.
Therefore
$$|N_{\Gamma^{\star}}(U_X^{\star})|+|N_{\Gamma}(U^{\perp}_X)\cap V^{\perp}|\geq (d-1)|U_X|.$$
Since $d\geq 2$, it follows that 
$$|N_{\Gamma}(X)|\geq (d-1)|X|\geq \frac{1}{d}|X|.$$

\end{proof}

The conditions of this lemma do not fit to the situation of Part 1 of the step of the main construction. 
Indeed, when some vertices from the graph are removed, we cannot guarantee that it is still $U$-reflected. 
On the other hand, it can happen that for some $v\in V$, the subgraph $\Gamma(-v)$ is $U$-reflected. 
This is the reason why in the following lemmas both $\Gamma^{\perp}$ and $\Gamma^{\star}$ are  subgraphs of some $\Gamma(-v)$.

\begin{lm}\label{neighsize2}
Let $v$ be a vertex from $V$ and 
$\Gamma(-v)=(U,V(-v))$. 
Assume that $\Gamma(-v)$ is $U$-reflected.
Let $\Gamma^{\perp}=(U^{\perp},V^{\perp},E^{\perp})$ be a subgraph of $\Gamma(-v)$ and let $\Gamma^{\star}$ be defined as the subgraph of $\Gamma(-v)$ induced by the sets of vertices 
$U^{\star}:=U\setminus U^{\perp},V^{\star}:=V(-v)\setminus V^{\perp}$.

Assume that $\Gamma^{\star}$ satisfies Hall's $d$-harem condition, and assume that for each $Y\subset U^{\perp}$ we have 
$|N_{\Gamma}(Y)\cap V^{\perp}|\geq (d-1)|Y|$.

Then for any $X \subseteq V\setminus\{v\}$ the inequality 
$|N_{\Gamma}(X)|\geq (d-1)|X|-1$ holds. \\ 
Moreover, the equality $|N_{\Gamma}(X)|=(d-1)|X|-1$ can happen only if $v\in N_{\Gamma}(U_X)$ and 
$N_{\Gamma^{\star}}(U_X)=\emptyset$, 
where $U_X =\{u\in U: v_u\in X\}$.  

\end{lm}
\begin{proof}
By $U$-reflectedness of $\Gamma(-v)$ we have $|U_X|=|X|$ and 
$$
|N_{\Gamma}(X)|+1\geq|N_{\Gamma(-v)}(U_X)|.
$$
Furthermore, if $v\notin N_{\Gamma}(U_X)$, then 
$$
|N_{\Gamma}(X)|\geq|N_{\Gamma(-v)}(U_X)|.
$$ 
Consider the sets 
$$
U_X^{\star}:=U_X\setminus U^{\perp} \, \mbox{ and } \, 
U_X^{\perp}:=U_X\cap U^{\perp}.
$$
Applying the argument of Lemma \ref{neighsize} to $\Gamma (-v)$ we obtain 
$$
|N_{\Gamma(-v)}(U_X)| \geq |N_{\Gamma^{\star}}(U_X^{\star})|+|N_{\Gamma}(U^{\perp}_X)\cap V^{\perp}| \geq (d-1)|U_X|.
$$
Thus 
$|N_{\Gamma}(X)|\geq |N_{\Gamma(-v)}(U_X)|-1\geq (d-1)|X|-1$.

Observe that $|N_{\Gamma(-v)}(U_X)|=(d-1)|U_X|$ only if $N_{\Gamma(-v)}(U_X)\subseteq V^{\perp}$, i.e. $N_{\Gamma^{\star}}(U_X)=\emptyset$ (since Hall's $d$-harem condition is satisfied for any subset of $U^{\star}$). 
If $v\notin N_{\Gamma}(U_X)$, then 
$$|N_{\Gamma}(X)|\geq |N_{\Gamma(-v)}(U_X)| \geq (d-1)|X|.$$
Therefore $|N_{\Gamma}(X)|=(d-1)|X|-1$ only if $v\in N_{\Gamma}(U_X)$ and $N_{\Gamma^{\star}}(U_X)=\emptyset$.

\end{proof}

The following proposition will be used regularly in the proofs of Lemmas \ref{1st} and \ref{2nd}.

\begin{pr}[Removing at most $k$ fans from a ball]\label{hallincom}
Assume that $\Gamma=(U,V,E)$ is a highly computable bipartite graph satisfying $c.e.H.h.c.(d)$ with witness $h$.
Let: 
\begin{itemize} 
\item $u\in U$ and $\mathcal{B}(u)$ be the ball in $\Gamma$ with the center $u$ and the radius of at least \\ 
$\max\{2h(kd)+3,5\}$; 
\item  
$M_u$ be a $(1,d)$-matching in $\mathcal{B}(u)$
which satisfies the conditions of the perfect $(1,d)$-matchings for all vertices that are at the distance less than $\max\{2h(kd)+3,5\}$ from $u$; 
\item $U_1\subseteq U$ be a subset of the ball of $u$ of radius $\le 2$, 
$|U_1 |\le k$, and $V_1$ be the set of vertices adjacent in $M_u$ to vertices from $U_1$.  
\end{itemize} 
Then $\Gamma':=(U\setminus U_1,V\setminus V_1)$ satisfies Hall's $d$-harem condition.
\end{pr} 

\begin{proof}
Let us show that the inequality $|N_{\Gamma'}(X)|-d|X|\geq 0$ holds for any $X\subset U\setminus U_1$. 
If $X$ is not connected (in the sense of the Definition \ref{connected}), then its neighborhood is a disjoint union of the neighborhoods of the connected components of $X$. 
Thus the inequality for each component implies it for $X$. 
Therefore, without loss of generality we may also assume that $X$ is connected. 
Further, we may assume that $X\nsubseteq U\cap\mathcal{B}(u)$ and $V_1\cap N_{\Gamma}(X)\neq\emptyset$.
Indeed, if $V_1\cap N_{\Gamma}(X)=\emptyset$ then $N_{\Gamma'}(X)=N_{\Gamma}(X)$, i.e. $|N_{\Gamma'}(X)|\geq d|X|$.
On the other hand, if $X\subset U\cap\mathcal{B}(u)$ then using 
existence of an $(1,d)$-matching $M_u$ from $U\cap\mathcal{B}(u)$ to $V\cap\mathcal{B}(u)$ which is perfect for subsets of $U\cap\mathcal{B}(u)$, we see that $|N_{\Gamma'}(X)|-d|X|\geq 0
$. 

Let us choose $u_1,u_2\in X$ such that  $u_1\in N_{\Gamma}(V_1)$ and $u_2\in X\setminus \mathcal{B}(u)$. 
By the choice of the radius of the ball $\mathcal{B}(u)$ the distance between $u_1$ and $u_2$ is at least $2h(kd)+1$. 
Thus $|X|\geq h(kd)$. 
Since $\Gamma$ satisfies $c.e.H.h.c.(d)$, we have 
$$
|N_{\Gamma}(X)|-d|X|\geq kd.
$$
We know that $|N_{\Gamma'}(X)|\geq |N_{\Gamma}(X)|-kd$, therefore
$$
|N_{\Gamma'}(X)|-d|X|\geq|N_{\Gamma}(X)|-kd-d|X|\geq 0.
$$

The proof for the case of subsets of $V'$ is completely analogous.

\end{proof}

The following definition will be useful in the proofs of the lemmas of the next section.
\begin{df}\label{acc}
Let $\Gamma:=(U,V,E)$ be a bipartite graph.  Assume that 
$\Gamma^{\star}=(U^{\star},V^{\star},E^{\star})$ is a subgraph of $\Gamma$ satisfying Hall's $d$-harem condition, and $M$ is a perfect $(1,d)$-matching in $\Gamma^{\star}$.

Let $X$ be a subset of $V$ and let $x\in X$.
We say that $x$ \textit{ is accessible from } $y\in N_{\Gamma}(X)$ \textit{ through } $X$ \textit{ by matching } $M$, (denoted by 
$y \xhookdoubleheadrightarrow{M,X} x$) if there exist two sequences of vertices $\{v'_0,\ldots,v'_n\}\subseteq X$ and 
$\{u'_0,\ldots,u'_{n-1}\}\subseteq N_{\Gamma}(X)$ such that 
\begin{itemize}
\item $v'_n=x$;
\item $(u'_i,v'_{i})\in M$ and $(u'_{i},v'_{i+1})\in E\setminus M$, where $i <n$;
\item $(y,v'_0)\in E$.
\end{itemize}
\end{df}

\begin{figure}[ht]
  \centering
    \includegraphics[width=0.7\textwidth]{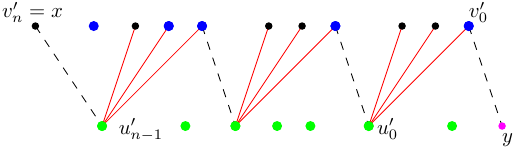}
	\caption{$x$ is accessible from $y$ by matching $M$ (red), the elements of $X$ are  blue and the elements of $N_{\Gamma}(X)$ are green}
\end{figure}

\begin{lm}\label{access}
Let $\Gamma=(U,V,E)$ be a bipartite graph. 
Assume that $\Gamma^{\star}=(U^{\star},V^{\star},E^{\star})$ is a subgraph of $\Gamma$ satisfying Hall's $d$-harem condition, and $M$ is a perfect $(1,d)$-matching in $\Gamma^{\star}$.

Assume that 
\begin{itemize}
\item $\widehat{v}\in N_{\Gamma}(U^{\star})\setminus V^{\star}$;
\item $X$ is a minimal connected subset in $V^{\star}\cup \{\widehat{v}\}$ with $|N_{\Gamma}(X)|< \frac{1}{d}|X|$;
\item $\widehat{u}\in N_{\Gamma}(X)\setminus U^{\star}$. 
\item $(\widehat{u},\widehat{v})\notin E$. 
\end{itemize}
Then $\widehat{u} \xhookdoubleheadrightarrow{M,X} \widehat{v}$.
\end{lm}

\begin{proof}
It is clear that $\widehat{v}\in X$. 
Let $X'$ denote the subset of all elements of $X$ that are either adjacent to $\widehat{u}$ or accessible from $\widehat{u}$ through $X$ by the matching $M$.  
Since $\widehat{u}\in N_{\Gamma}(X)$, $X'\neq \emptyset$. 
%%%
%we then see that there exist $u'_0,v'_0$ such %that $\widehat{u}, u'_0\in N_{\Gamma}(v'_0)$ 
%and $(u'_{0},v'_0)\in M$. 
%%%%
In order to get a contradiction, assume that $\widehat{v}\not\in X'$. 
We will show that 
$|N_{\Gamma}(X\setminus X')|< \frac{1}{d}|X\setminus X'|$. 

Let $|X'|=l$ and let 
$U_M (X')=\{ u\in U \, | \, (\exists v\in X') (u,v)\in M\}$. 
Since elements of $X\setminus X'$ are not accessible from $\widehat{u}$ through $X$ by the matching $M$, then 
$N_{\Gamma}(X \setminus X') \cap U_M (X') = \emptyset$. 

Using this we see that 
$$
|N_{\Gamma}(X\setminus X')|\leq |N_{\Gamma}(X)|-|U_M(X')|. 
$$
Since $M$ is a $(1,d)$-matching, each element of $U_M(X')$ can be matched with at most $d$ elements from $X'$. 
Therefore $|U_M(X')|\geq \lceil\frac{l}{d}\rceil$ and 
$$
|N_{\Gamma}(X\setminus X')|\leq |N_{\Gamma}(X)|-|U_M(X')|\leq |N_{\Gamma}(X) |-\lceil\frac{l}{d}\rceil<\frac{1}{d}(|X|-l)=\frac{1}{d}|X\setminus X'|. 
$$ 
So $|N_{\Gamma}(X\setminus X')|< \frac{1}{d}|X\setminus X'|$  and $X\setminus X'$ is smaller than $X$, a contradiction with the choice of the latter.
As a result $\widehat{v}\in X'$, and consequently $\widehat{u} \xhookdoubleheadrightarrow{M,X} \widehat{v}$.
\end{proof}

This lemma will be used in situations when some $(1,d-1)$-fans in the matching are replaced by some new fans.  
%Since in our  case this matching is partial and %finite we need an additional 
The following technical lemma will additionally guarantee that 
%in a graph satisfying $c.e.H.h.c.(d)$  
%%
Hall's $d$-harem condition can be preserved after some replacement of $(1,d-1)$-fans. 

\begin{lm}[Replacement of a fan in a finite matching]\label{hccom}
Let 
\begin{itemize}
\item $\Gamma=(U,V,E)$ be a highly computable bipartite graph satisfying $c.e.H.h.c.(d)$ with a witness $h$;

\item $\Gamma^{\star}=(U^{\star},V^{\star})$ be a subgraph of $\Gamma$ such that $\Gamma\setminus \Gamma^{\star}$ consists of $n$ $(1,d-1)$-fans;

\item $u\in U$ and $\mathcal{B}(u)$ be a ball in $\Gamma$ of $u$ of odd radius $\ge \max\{4h(d(n+1))+3,5\}$;

\item $\mathfrak{F}_u$ be an $(1,d-1)$-fan in $\Gamma^{\star}$ with the root $u$;

\item $\mathfrak{S}$ be an $(1,d-1)$-fan in $\Gamma \setminus \Gamma^{\star}$ with the root $u^{\perp}$; 

\item $\mathfrak{M}$ be an 
$(1,d)$-matching in the ball $((\Gamma^{\star}\setminus \mathfrak{F}_u )\cup \mathfrak{S})\cap \mathcal{B}(u)$ 
which satisfies the conditions of perfect $(1,d)$-matchings for all vertices that are in $\mathcal{B}(u)\setminus \mathcal{S}(u)$. 
\end{itemize} 

Assume that $\Gamma^{\star}$ satisfies 
$c.e.H.h.c.(d)$ with the witness 
$$
\widehat{h}(m) = 
\left\{
\begin{array}{rr}
0, \quad \text{if} \quad m=0, \ \\
h(m+nd), \quad \text{if} \quad m > 0.
\end{array}\right.
$$
If $\Gamma^{\star}\setminus \mathfrak{F}_u$ does not satisfy Hall's $d$-harem condition, 
then $(\Gamma^{\star}\setminus \mathfrak{F}_u) \cup \mathfrak{S}$ already satisfies Hall's $d$-harem condition.
\end{lm}

\begin{proof}
The proof contains several claims. 
\begin{clm} \label{dist0} 
All subsets of $U^{\star}\setminus \{u\}$ satisfy Hall's $d$-harem condition in $\Gamma^{\star}\setminus \mathfrak{F}_u$. 
\end{clm} 
\begin{proof}[Proof of Claim \ref{dist0}] 
Assume that $X\subset U^{\star}\setminus \{u\}$ does not satisfy Hall's $d$-harem condition. 
If $X$ is not connected, then its neighborhood is a disjoint union of the neighborhoods of the connected components of $X$. 
Therefore, without loss of generality, we may assume that $X$ is connected. 
Since $\Gamma^{\star}$ satisfies Hall's $d$-harem condition,  there exists $u'\in X$ such that the distance between $u'$ and $u$ is equal to $2$.
On the other hand, the matching $\mathfrak{M}$ witnesses Hall's $d$-harem condition for all subsets of $((U^{\star}\setminus \mathfrak{F}_u )\cup \{u^{\perp}\})\cap\mathcal{B}(u)$.
Therefore 
$X\not\subseteq ((U^{\star}\setminus \{u\})\cup \{u^{\perp}\})\cap\mathcal{B}(u)$.  
It follows that there exists $u''\in X\setminus \mathcal{B}(u)$. 
The distance between $u'$ and $u''$ is at least $4h(d(n+1))+1$, i.e. $|X|\geq h(d(n+1))+1$. 
By the definition of $\hat{h}$,  
$$
|N_{\Gamma^{\star}}(X)|\geq d|X|+d. 
$$
Thus 
$$
|N_{\Gamma^{\star}\setminus \mathfrak{F}_u}(X)|\geq |N_{\Gamma^{\star}}(X)|-(d-1)\geq d|X|+1,
$$ 
i.e. $X$ satisfies Hall's $d$-harem condition in $\Gamma^{\star}\setminus \mathfrak{F}_u$.
\end{proof} 

Assume that $\Gamma^{\star}\setminus \mathfrak{F}_u$ does not satisfy Hall's $d$-harem condition.  
By this claim, there is $X\subseteq V^{\star}\setminus \mathfrak{F}_u$ that does not satisfy Hall's $d$-harem condition in $\Gamma^{\star}\setminus \mathfrak{F}_u$. 
We fix a connected one that is minimal. 
Observe that $u$ must be in $N_{\Gamma^{\star}}(X)$.
Indeed, otherwise $|N_{(\Gamma^{\star}\setminus \mathfrak{F}_u)}(X)|=|N_{\Gamma^{\star}}(X)|$. 
Since $\Gamma^{\star}$ satisfies $c.e.H.h.c.(d)$ we get a contradiction.

\begin{clm}\label{dist1}
For any vertex $v$ from $X$ the distance between $u$ and $v$ cannot exceed $2h(d(n+1))-1$.
\end{clm}
\begin{proof}[Proof of Claim \ref{dist1}]
Since $\Gamma$ is bipartite, then distance between $u$ and any element of the set $X$ is odd.
If for some $v\in X$ the distance between $u$ and $v$ is at least $2h(d(n+1))+1$, then $X$ has at least $h(d(n+1))$ elements. Since $\widehat{h}$ is a witness of $c.e.H.h.c.(d)$ for $\Gamma^{\star}$, we have $$|N_{\Gamma^{\star}\setminus\mathfrak{F}_u}(X)|\geq|N_{\Gamma^{\star}}(X)|-1\geq \frac{1}{d}|X|+1-1\geq \frac{1}{d}|X|,$$ which contradicts that $X$ does not satisfy Hall's $d$-harem condition.
\end{proof}
As a consequence of Claim \ref{dist1}, $X\subset(\mathcal{B}(u)\setminus \mathcal{S}(u))$.
The following claim, together with the above observations, guarantees that $\{u,u^{\perp}\}\subset N_{\Gamma^{\star}\cup \mathfrak{S}}(X)$. 

\begin{clm}\label{dist2}
The element $u^{\perp}$ belongs to $N_{(\Gamma^{\star}\setminus \mathfrak{F}_u)\cup \mathfrak{S}}(X)$ and the distance between $u^{\perp}$ and $u$ is at most $2h(d(n+1))$.
\end{clm}

\begin{proof}[Proof of Claim \ref{dist2}]
Since $X\subset(\mathcal{B}(u)\setminus \mathcal{S}(u))$ and $\mathfrak{M}$ satisfies the conditions of perfect $(1,d)$-matchings for all subsets of $((\Gamma^{\star}\setminus \mathfrak{F}_u )\cup \mathfrak{S}) \cap (\mathcal{B}(u)\setminus \mathcal{S}(u))$, then 
$$
|N_{(\Gamma^{\star}\setminus \mathfrak{F}_u)\cup \mathfrak{S}}(X)|\geq \frac{1}{d}|X|>|N_{\Gamma^{\star}\setminus\mathfrak{F}_u}(X)|.
$$ 
Clearly $u^{\perp}$ is in $N_{(\Gamma^{\star}\setminus \mathfrak{F}_u)\cup \mathfrak{S}}(X)$.

If the distance between $u^{\perp}$ and $u$ is at least $2h(nd+1)+2$ then there exists an element of $X$ which is at distance at least $2h(nd+1)+1$ from $u$. 
Contradiction with Claim \ref{dist1}.
\end{proof}

Let us prove that Hall's $d$-harem condition is satisfied in $(\Gamma^{\star}\setminus \mathfrak{F}_u) \cup \mathfrak{S}$ for subsets of $(U^{\star}\setminus\{u\})\cup\{u^{\perp}\}$.
By Claim \ref{dist0} all subsets of $U^{\star}\setminus \{u\}$ satisfy Hall's $d$-harem condition in $\Gamma^{\star}\setminus \mathfrak{F}_u$ and so in $(\Gamma^{\star}\setminus \mathfrak{F}_u) \cup \mathfrak{S}$ too.
Therefore, it remains to consider the case of subsets of $(U^{\star}\setminus\{u\})\cup\{u^{\perp}\}$ that contain $u^{\perp}$.
Trying to verify Hall's $d$-harem condition in $(\Gamma^{\star}\setminus \mathfrak{F}_u) \cup \mathfrak{S}$ for  $Y\subset (U^{\star}\setminus\{u\})\cup\{u^{\perp}\}$ we may assume that $Y\nsubseteq \mathcal{B}(u)$.  
Indeed, otherwise the matching $\mathfrak{M}$ already witnesses Hall's $d$-harem condition for all subsets of $(\Gamma^{\star}\setminus \mathfrak{F}_u)\cup \mathfrak{S}$ contained in $\mathcal{B}(u)\setminus \mathcal{S}(u)$.
As above, we also assume that $Y$ is connected.

Let $u_1 \in Y\setminus \mathcal{B}(u)$.  
The distance between $u$ and $u_1$ is at least $4h(d(n+1))+4$.
Since $u^{\perp} \in N_{\Gamma^{\star}\cup \mathfrak{S}}(X)\cap Y$, we see by Claim \ref{dist2} that there is some $u_2 \in Y$ at the distance $\le 2h(d(n+1))$ from $u$. 
Then the distance between $u_1$ and $u_2$ is at least $4h(d(n+1))+4-(2h(d(n+1))+2)$, i.e at least $2h(d(n+1))+2$.

\begin{figure}[ht]\label{fig:sets}
   \centering
   \includegraphics[width=0.7\textwidth]{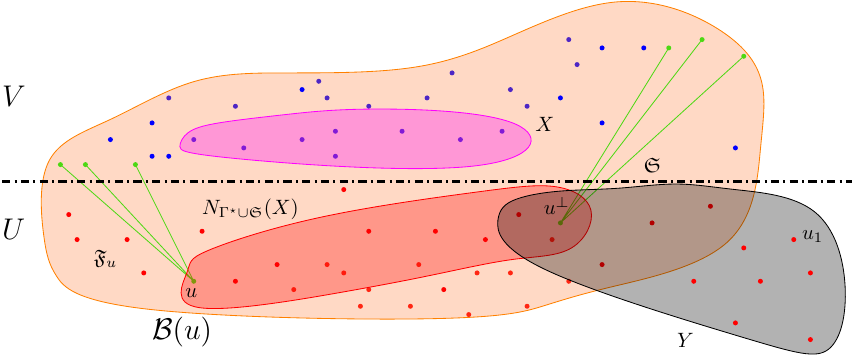}
	\caption{}
\end{figure}

Thus $|Y| > h(d(n+1))$, and 
$$
|N_{\Gamma}(Y)|-\frac{1}{d}|Y|\geq d(n+1).
$$
On the other hand,
$$
|N_{(\Gamma^{\star}\setminus \mathfrak{F}_u)\cup\mathfrak{S}}(Y)|\geq |N_{\Gamma}(Y)|-n(d-1), 
$$ 
hence:
$$
|N_{(\Gamma^{\star}\setminus \mathfrak{F}_u)\cup \mathfrak{S}}(Y)|-\frac{1}{d}|Y|\geq |N_{\Gamma}(Y)|-\frac{1}{d}|Y|-n(d-1)\geq n\geq 0.
$$
It follows that Hall's $d$-harem condition holds for any $Y\subset (U^{\star}\setminus\{u\})\cup\{u^{\perp}\}$ such that $u^{\perp}\in Y$. 
This finishes the case of subsets of $(U^{\star}\setminus\{u\})\cup\{u^{\perp}\}$.

It remains to show that Hall's $d$-harem condition is satisfied in $(\Gamma^{\star}\setminus \mathfrak{F}_u) \cup \mathfrak{S}$ for every $Y\subset V\cap ((\Gamma^{\star}\setminus \mathfrak{F}_u)\cup \mathfrak{S})$.
We may assume that $N_{\Gamma^{\star}\cup \mathfrak{S}}(Y)\cap \{u,u^{\perp}\} \neq \emptyset$. 
Indeed, otherwise apply the assumption that $\Gamma^{\star}$ satisfies Hall's $d$-harem condition.  
Also note again that the matching $\mathfrak{M}$ witnesses Hall's $d$-harem condition for all subsets of $(\Gamma^{\star}\setminus \mathfrak{F}_u)\cup \mathfrak{S}$ contained in $\mathcal{B}(u)\setminus \mathcal{S}(u)$. 
In particular, we may assume that $Y\nsubseteq (\mathcal{B}(u)\setminus \mathcal{S}(u))$. 
As above, we only consider connected $Y\subset (V\cap ((\Gamma^{\star}\setminus \mathfrak{F}_u)\cup \mathfrak{S})$.   

Take any element $v_1$ of $Y\setminus (\mathcal{B}(u)\setminus \mathcal{S}(u))$. 
It is clear that the distance from $u$ to $v_1$ is at least $4h(d(n+1))+3$. 
Recall our assumption that one of the elements $u,u^{\perp}$ is in $N_{\Gamma^{\star}\cup \mathfrak{S}}(Y)$.
If $u\in N_{\Gamma^{\star}\cup \mathfrak{S}}(Y)$ then $Y$ contains a vertex at the distance $1$ from $u$. 
If $u^{\perp}\in N_{\Gamma^{\star}\cup \mathfrak{S}}(Y)$ then by Claim \ref{dist2} $Y$ contains a vertex at the distance $2h(d(n+1))+1$ from $u$. 
In either case we can find $v_2\in Y$ at the distance at most $2h(d(n+1))+1$ from $u$. 
We see that the distance between $v_1$ and $v_2$ is at least $2h(d(n+1))+2$.
\begin{figure}[ht]\label{fig:sets1}
  \centering
    \includegraphics[width=0.7\textwidth]{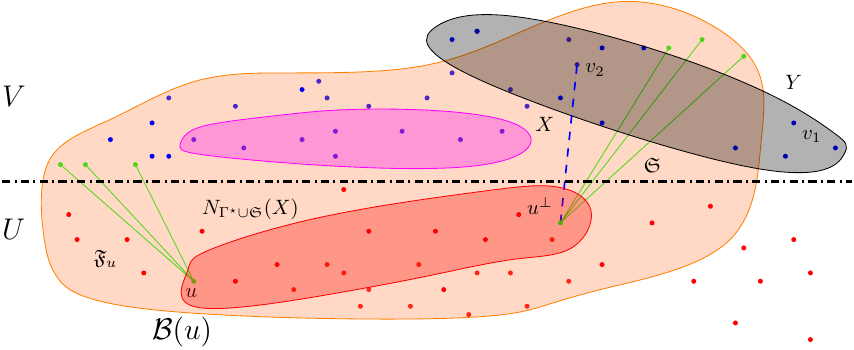}
	\caption{}
\end{figure}
Thus, $Y$ has at least $h(d(n+1))$ elements and  
$$
|N_{\Gamma(Y)}|-\frac{1}{d}|Y|\geq d(n+1).
$$
On the other hand, using 
$$
|N_{\Gamma^{\star}\cup \mathfrak{S}\setminus \mathfrak{F}_u}(Y)|\geq|N_{\Gamma(Y)}|-n 
$$ 
we see that:
$$
|N_{\Gamma^{\star}\cup \mathfrak{S}\setminus \mathfrak{F}_u}(Y)|- \frac{1}{d}|Y|\geq |N_{\Gamma(Y)}|-n - \frac{1}{d}|Y|\geq d(n+1)-n=n(d-1)+d\geq 0, 
$$
i.e. Hall's $d$-harem condition is also satisfied.
\end{proof}

The final lemma of this section shows how to find  a witness of $c.e.H.h.c.(d)$ in a subgraph.

\begin{lm}\label{chc}
Assume that $\Gamma=(U,V,E)$ is a highly computable bipartite graph satisfying $c.e.H.h.c.(d)$ with witness $h$.
Let 
\begin{itemize} 
\item $\mathfrak{M}$ be an $(1,d-1)$-matching in $\Gamma$ consisting of $\ell$ fans, 
\item $\Gamma^{\star} =(U^{\star},V^{\star})$ be the induced subgraph of $\Gamma$ without the vertices of $\mathfrak{M}$,  
\item $\Gamma^{\star}$ satisfy Hall's $d$-Harem condition.
\end{itemize} 
Then 
$$
\widehat{h}(m) = 
\left\{
\begin{array}{rr}
0, \quad \text{if} \quad m=0, \ \\
h(m+\ell d), \quad \text{if} \quad m > 0.
\end{array}\right.
$$
is a witness of $c.e.H.h.c.(d)$ for $\Gamma^{\star}$.
\end{lm}

\begin{proof} 

Since $\Gamma^{\star}$ satisfies Hall's $d$-harem condition, the case of $m=0$ is obvious. 
Assume that  $m>0$.  
First, we prove that $\widehat{h}$ is a witness of $c.e.H.h.c.(d)$ for subsets of $U^{\star}$. 
The inequality $|X|\geq\widehat{h}(m)$ implies $|X|\geq h(m+\ell d)$, i.e.   
$$
|N_{\Gamma}(X)|-d|X| -\ell d\geq m.
$$
Since $d\geq 2$ and 
$$
|N_{\Gamma^{\star}}(X)|\geq |N_{\Gamma}(X)|-\ell (d-1), 
$$
we have
$$
|N_{\Gamma^{\star}}(X)|-d|X|\geq |N_{\Gamma}(X)|-\ell (d-1)-d|X|\geq m.
$$

In order to show that $\widehat{h}$ works for $V^{\star}$ we start with the observation that for each $X\subset V^{\star}$, 
$$
|N_{\Gamma^{\star}}(X)|\geq |N_{\Gamma}(X)|-\ell.
$$ 
Therefore, using $d\geq 2$, we see that when 
$|X|\geq \widehat{h}(m)$,  
$$
|N_{\Gamma}(X)|-\frac{1}{d}|X|-m-\ell d\geq 0.
$$
It follows that: 
$$
|N_{\Gamma^{\star}}(X)|-\frac{1}{d}|X|\geq |N_{\Gamma}(X)|-\frac{1}{d}|X|-\ell d\geq m.
$$ 
\end{proof}

\section{Graphs constructed in parts 1 and 2 of each step satisfy \textit{c.e.H.h.c.(d)}}\label{lemmas}

\subsection{Claims  \ref{c1p1ef}, \ref{csnp1ef}}
Before stating Lemma \ref{1st} (proving claims \ref{c1p1ef} and \ref{csnp1ef}) we recall some notation used in case (2) of the $n+1$-st step of the construction.

\begin{itemize}
\item $\Gamma^{(n)}$ is $U^{(n)}$-reflected;
\item $\Gamma^{(n)\star}$ is a subgraph of $\Gamma^{(n)}$ obtained by removal of $(1,d-1)$-fans with roots belonging to the set $U^{(n)\perp}$; 
\item $\Gamma^{(n)\star}$ satisfies $c.e.H.h.c.(d)$, and $u_n \in \Gamma^{(n)\star}$;
\item $|U^{(n)\perp}|\leq n$;
\item $\mathfrak{M}^1_n$ is a finite $(1,d)$-matching in the bipartite graph 
$\mathcal{B}^{(n)\star}(u_n)=\mathcal{B}^{(n)}(u_n)\cap \Gamma^{(n)\star}$. 
It satisfies the conditions of the perfect $(1,d)$-matchings for all vertices that are at the distance less than $\max\{4h(3d(n+1))+3,5\}$ from $u_n$.
\end{itemize}
 
\begin{lm}\label{1st}
The bipartite graph $\Gamma^{(n)\star}(-u_n)$ is highly computable for any $n$.   
Furthermore, one of the following holds:
\begin{itemize}
\item  $\Gamma^{(n)\star}(-u_n)$ satisfies $c.e.H.h.c.(d)$;
\item the graph $\Gamma^{(n)\star}(-u_n,+u_j^{\perp})$ satisfies $c.e.H.h.c.(d)$ for some vertex $u_j^{\perp}\in U^{(n)\perp}$.
\end{itemize}
In the latter case the element $u^{\perp}_j$ can be computed and the bipartite graph $\Gamma^{(n)\star}(-u_n,+u_j^{\perp})$ is highly computable. 
\end{lm}

\begin{rem}
If $U^{(n)\perp}=\emptyset$, then Lemma \ref{1st} implies that the graph $\Gamma^{(n)\star}(-u_n)$ satisfies $c.e.H.h.c.(d)$. 
In particular, Claim \ref{c1p1ef} holds.
\end{rem}

\begin{proof}[Proof of Lemma \ref{1st}] 
Let 
$$
\widehat{h}_1(m) = \left\{
\begin{array}{rr}
0, \quad \text{if} \quad m=0, \ \\
h(m+(3n+1)d), \quad \text{if} \quad m > 0.
\end{array}\right.
$$
and 
$$
\widehat{h}_2(m) = \left\{
\begin{array}{rr}
0, \quad \text{if} \quad m=0, \ \\
h(m+3nd), \quad \text{if} \quad m > 0.
\end{array}\right.
$$

\begin{clm}\label{hall1}
Assume that one of the graphs $\Gamma^{(n)\star}(-u_n)$ or $\Gamma^{(n)\star}(-u_n,+u_j^{\perp})$ satisfies Hall's $d$-harem conditions. 
Then in the first case the function $\widehat{h}_1$ witnesses $c.e.H.h.c.(d)$ for $\Gamma^{(n)\star}(-u_n)$ and in the second case the function $\widehat{h}_2$ witnesses $c.e.H.h.c.(d)$ for $\Gamma^{(n)\star}(-u_n,+u_j^{\perp})$.
\end{clm}
\begin{proof}
Apply Lemma \ref{chc}. 
In the first case, view $\Gamma^{(0)}$ as $\Gamma$ and $\Gamma^{(n)\star}(-u_n)$ as $\Gamma^{\star}$ in this lemma. 
The number $\ell$ from the lemma is equal to the number of fans both in $\bigcup\limits_{i=1}^{n}M_{i-1}$ and in  $\Gamma^{(n)\perp}$, plus one, for the fan corresponding to $u_n$. 
At every step, at most one new fan appears in $\Gamma^{(n)\perp}$ and at most two new fans appear in $M_{n} \setminus\Gamma^{(n-1)\perp}$. 
Thus $\ell \le 3n+1$. 
Since $\Gamma^{(n)\star}(-u_n)$ satisfies Hall's $d$-harem conditions, the lemma works.

In the second case, view $\Gamma^{(0)}$ as $\Gamma$ and $\Gamma^{(n)\star}(-u_n,+u_j^{\perp})$ as $\Gamma^{\star}$. 
To compute the number corresponding to $\ell$ from the lemma, we should take the number of fans in $\bigcup\limits_{i=1}^{n}M_{i-1}$ and $\Gamma^{(n)\perp}$ together, then increase it by one, corresponding to the fan of $u_n$, and then subtract one, corresponding to the fan of $u_j^{\perp}$. 
Therefore, $\ell \le 3n$.  
Since $\Gamma^{(n)\star}(-u_n,+u_j^{\perp})$ satisfies Hall's $d$-harem conditions, Lemma \ref{chc} works again.
\end{proof}

It is clear that both $\Gamma^{(n)\star}(-u_n)$ and $\Gamma^{(n)\star}(-u_n,+u_j^{\perp})$ are highly computable bipartite graphs (assuming that $u^{\perp}_j$ is computed).
Thus, it is enough to show that one of these graphs satisfies Hall's $d$-harem condition (then apply the claim).
We know that the graph $\Gamma^{(n)\star}$ satisfies Hall's $d$-harem condition. 
Let $\mathfrak{v}$ denote the only vertex from the set $\{v^0_{n,1},\ldots ,v^0_{n,d}\}$ that belongs to $\Gamma^{(n)\star}(-u_n)$.
By Proposition \ref{hallincom} the choice of $u_n, v^0_{n,1},\ldots ,v^0_{n,d}$ ensures that $\Gamma^{(n)\star}(-u_n,-\mathfrak{v})$ satisfies Hall's $d$-harem condition as well. 
For further arguments: 

$\bullet$ let $\mathfrak{M}$ denote a perfect $(1,d)$-matching in $\Gamma^{(n)\star}(-u_n,-\mathfrak{v})$. 

Since $U^{(n)\star}(-u_n)=U^{(n)\star}(-u_n,-\mathfrak{v})$, for $X\subset U^{(n)\star}(-u_n)$ we have 
\begin{equation}\label{Uhcond}
|N_{\Gamma^{(n)\star}(-u_n)}(X)|\geq |N_{\Gamma^{(n)\star}(-u_n,-\mathfrak{v})}(X)|\geq d|X|.
\end{equation}

The corresponding property also holds for all subsets of $V^{(n)\star}(-u_n)$ that do not contain $\mathfrak{v}$. 
Therefore if $\Gamma^{(n)\star}(-u_n)$ does not satisfy Hall's $d$-harem condition, then a witness of this is a finite 
$X \subset V^{(n)\star}(-u_n)$ which contains $\mathfrak{v}$.

\begin{clm}\label{ins}
If $X$ is a connected subset of $V^{(n)\star}(-u_n)$ and $
|N_{\Gamma^{(n)\star}(-u_n)}(X)|< \frac{1}{d}|X|$, then $X\subset(\mathcal{B}(u_n)\setminus \mathcal{S}(u_n))$.
\end{clm}
\begin{proof}
Observe that by the choice of the radius of the ball $\mathcal{B}(u_n)$, 
for any connected $X\nsubseteq (\mathcal{B}(u_n)\setminus \mathcal{S}(u_n))$ with $\mathfrak{v}\in X$ we have  $|X|\geq 2h(d(3n+1))+1$. 
Therefore 
$$
|N_{\Gamma^{(0)}}(X)|-\frac{1}{d}|X|\geq d(3n+1).
$$
On the other hand at every step at most 3 fans from $\Gamma$ are added to the matching, hence $|U^{(0)}\setminus U^{(n)\star}(-u_n)|\leq 3n+1$. As a consequence 
$$
|N_{\Gamma^{(n)\star}(-u_n)}(X)|\geq |N_{\Gamma^{(0)}}(X)|-(3n+1)\geq |N_{\Gamma^{(0)}}(X)|-d(3n+1)\geq \frac{1}{d}|X|.
$$
\end{proof}

Since the neighborhood of any set is decomposed into a disjoint union of the neighborhoods of connected subsets,  applying the claim, we see that if $\Gamma^{(n)\star}(-u_n)$ does not satisfy Hall's $d$-harem condition, then this is witnessed by a finite connected $X \subset V^{(n)\star}(-u_n) \cap(\mathcal{B}(u_n)\setminus\mathcal{S}(u_n))$ which contains $\mathfrak{v}$. 
Using high computability of $\Gamma^{(n)}$ and finiteness of $\mathcal{B}(u_n)\setminus \mathcal{S}(u_n)$, we compute $X$ with these properties; and, furthermore, it can be taken to be minimal.  
We will now prove the lemma into two steps.  \\ 
$\bullet$ {\em Step 1. There exists some} $u_j^{\perp}\in N_{\Gamma^{(n)}(-u_n)}(X)$. \\  
This is the main part of the proof of the lemma. 
In fact we will prove that 
\begin{equation}\label{0}
|N_{\Gamma^{(n)}(-u_n)}(X)|\geq \frac{1}{d}|X|.
\end{equation} 
Then the existence of the required $u^{\perp}_j$ follows from the inequality $|N_{\Gamma^{(n)\star}(-u_n)}(X)|<\frac{1}{d}|X|$. 

First, consider the case when 
$X=\{\mathfrak{v}\}$. 
The vertex $u_{\mathfrak{v}}$ either belongs to $U^{(n)\star}(-u_n)$ or to $U^{(n)\perp}$. 
By inequality (\ref{Uhcond}) and the fact that $U^{(n)\perp}$ consist of $(1,d-1)$-fans, in each case we have $|N_{\Gamma^{(n)}(-u_n)}(u_{\mathfrak{v}})|\geq d-1$.
Moreover, if the equality
\begin{equation}\label{11}
|N_{\Gamma^{(n)}(-u_n)}(u_{\mathfrak{v}})|= d-1
\end{equation} 
holds, then $v_{u_n}\not\in \Gamma^{(n)}$. 
Indeed, if $v_{u_n} \in \Gamma^{(n)}$, then $v_{u_n}\in N_{\Gamma^{(n)}(-u_n)}(u_{\mathfrak{v}})$ (by reflectedness).
On the other hand, the existence of $\mathfrak{M}$ together with equality (\ref{11}) implies that $(u_{\mathfrak{v}},v_{u_n})\in\Gamma^{(n)\perp}$. 
However, the first part of the $n+1$-st step of the construction implies that $(u_n,\mathfrak{v})$ should be added to the matching $M_n$ already: the element $\mathfrak{v}$ plays the role of $v^0_{n,j}$ in case (2) of the first part. 
In particular, $\mathfrak{v}\notin \Gamma^{(n)\star}(-u_n)$, a contradiction.

Our next observation is 
$$
|N_{\Gamma^{(n)}(-u_n)}(\mathfrak{v})| \geq |N_{\Gamma^{(n)}(-u_n)}(u_\mathfrak{v})|-1. 
$$
This follows by $U^{(n)}$-reflectedness of $\Gamma^{(n)}$: $v_{u_n}$ is the only possible element adjacent to $u_{\mathfrak{v}}$ that does not have the left copy in $U^{(n)}(-u_n)$.
Additionally, note that the equality 
\begin{equation}\label{22}
|N_{\Gamma^{(n)}(-u_n)}(\mathfrak{v})|= |N_{\Gamma^{(n)}(-u_n)}(u_\mathfrak{v})|-1
\end{equation}
holds only if $v_{u_n}\in \Gamma^{(n)}(-u_n)$.
 
Therefore equalities (\ref{11}), (\ref{22}) are not consistent, i.e.:
$$
|N_{\Gamma^{(n)}(-u_n)}(\mathfrak{v})|	 >  (d-1)-1 \mbox{ and, in particular, } 
|N_{\Gamma^{(n)}(-u_n)}(\mathfrak{v})| \geq \frac{1}{d}.
$$ 
Since $X$ is a singleton,
$$
|N_{\Gamma^{(n)}(-u_n)}(X)|\geq \frac{1}{d}|X|.
$$
As a result we have inequality (2) and the fact  that there is some 
$u_j^{\perp}\in N_{\Gamma^{(n)}(-u_n)}(X)$.  

Consider this case when $X\not=\{ \mathfrak{v}\}$. Then $|N_{\Gamma^{(n)\star}(-u_n)}(X)|\geq 1$ and 
by the assumption 
$|N_{\Gamma^{(n)\star}(-u_n)}(X)|<\frac{1}{d}|X|$, we have $|X|>d$. 

Let $U_X:=\{u\in U: v_u\in X\}$ in $\Gamma^{(n)}(-u_n)$. 
Applying the fact that $V^{(n)}$ is a subset of the right copy of $U^{(n)}$ we arrive at two possibilities:
\begin{enumerate}[(i)]
\item $v_{u_n}\notin X$ and  $|X|=|U_X|$;
\item $v_{u_n}\in X$ and $|X|=|U_X|+1$.
\end{enumerate} 
We claim that in either case the inequality (\ref{0}) follows from Lemma \ref{neighsize2}.

\begin{figure}[ht]
  \centering
    \includegraphics[width=0.7\textwidth]{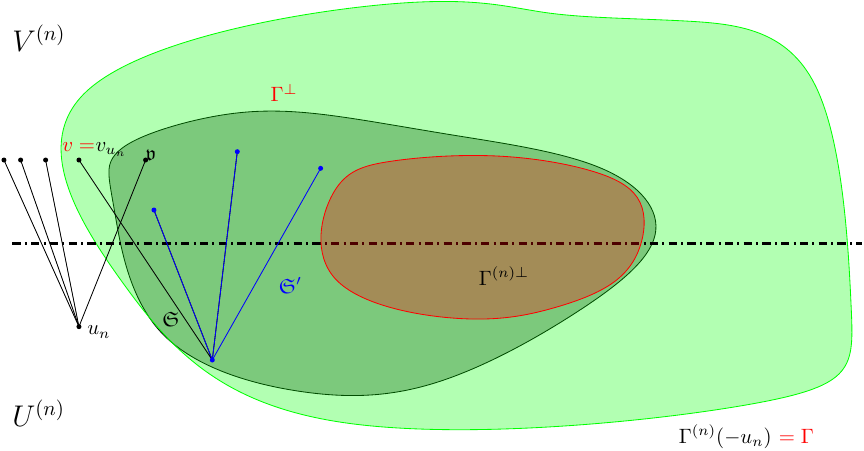}
	\caption{Application of Lemma \ref{neighsize2}; red letters correspond to the  notation of Lemma \ref{neighsize2}.}
\end{figure}

Indeed, let $\mathfrak{S}$ denote the fan from the matching $\mathfrak{M}$ containing $v_{u_n}$ and $\mathfrak{S}'$ denote $\mathfrak{S}$ with $v_{u_n}$ removed.
The conditions of Lemma \ref{neighsize2} are satisfied if we consider $\Gamma^{(n)}(-u_n)$ as $\Gamma$ in that lemma, $v_{u_n}$ as $v$, and 
$\Gamma^{(n)\perp}\cup\{\mathfrak{v}\}\cup \mathfrak{S}'$ as $\Gamma^{\perp}$. 
Indeed, $\Gamma^{(n)}(-u_n,-v_{u_n})$ is $U^{(n)}(-u_n,-v_{u_n})$-reflected. 
Moreover, the corresponding graph $\Gamma^{\star}$ from Lemma \ref{neighsize2} is obtained by removal of $\Gamma^{(n)\perp}\cup\{\mathfrak{v}\}\cup\mathfrak{S}'$ from $\Gamma^{(n)}(-u_n,-v_{u_n})$, therefore it is the same as $\Gamma^{(n)\star}(-u_n,-\mathfrak{v})$ with $\mathfrak{S}$ removed. 
Since $\mathfrak{S}$ is a fan from a perfect $(1,d)$-matching in $\Gamma^{(n)\star}(-u_n,-\mathfrak{v})$, we know that $\Gamma^{\star}$ satisfies Hall's $d$-harem condition.

Therefore in case (i), by Lemma \ref{neighsize2}:
$$|N_{\Gamma^{(n)}(-u_n)}(X)|\geq (d-1)|X|-1.$$
This fact combined with inequalities $d\geq 2$ and $|X|>d$ implies that 
$$|N_{\Gamma^{(n)}(-u_n)}(X)|\geq\frac{1}{d}|X|.$$
In case (ii), by Lemma \ref{neighsize2}:
$$|N_{\Gamma^{(n)}(-u_n)}(X\setminus \{v_{u_n}\})|\geq (d-1)(|X|-1)-1, 
$$ 
$$
\mbox{ i.e. } \, |N_{\Gamma^{(n)}(-u_n)}(X)| \geq (d-1)(|X|-1)-1.
$$
Note that the inequality is strict. 
Indeed, by Lemma \ref{neighsize2} the equality 
$$
|N_{\Gamma^{(n)}(-u_n)}(X\setminus \{v_{u_n}\})|=(d-1)(|X|-1)-1
$$
implies that 
$v_{u_n}\in N_{\Gamma^{(n)}(-u_n)}(U_{X\setminus \{v_{u_n}\}})$ and $N_{\Gamma^{(n)\star}(-u_n,-\mathfrak{v})}(U_{X\setminus \{v_{u_n}\}})=\emptyset$.
Thus, $v_{u_n}\in V^{(n)\perp}$. 
On the other hand, $v_{u_n}\in X\subset \Gamma^{(n)\star}$, a contradiction with the choice of $X$.

Again, inequalities $d\geq 2$ and $|X|>d$ imply 
$$|N_{\Gamma^{(n)}(-u_n)}(X)|\geq(d-1)(|X|-1)\geq\frac{1}{d}|X|.$$ 
This finishes our argument for the statement that the assumption 
$|N_{\Gamma^{(n)\star}(-u_n)}(X)|<\frac{1}{d}|X|$ implies 
$$
N_{\Gamma^{(n)}(-u_n)}(X)\cap U^{(n)\perp}(-u_n)\neq\emptyset , 
$$
i.e. there exists some $u_j^{\perp}\in N_{\Gamma^{(n)}(-u_n)}(X)$. 

The vertex $u_j^{\perp}$ that we have found, belongs to $N_{\Gamma^{(n)}(-u_n)\cap\mathcal{B}(u_n)}(X)$ (recall that $X\subseteq \mathcal{B}(u_n)\setminus \mathcal{S}(u_n)$). 
Thus, it can be found effectively by high computability of $\Gamma^{(n)}(-u_n)$.
Let us denote by $v^{\perp}_{j,1},\ldots, v^{\perp}_{j,d-1}$ 
the $(d-1)$-tuple of vertices adjacent to $u_j^{\perp}$ in $\Gamma^{(n)\perp}$. \\ 
$\bullet$ {\em Step 2. $\Gamma^{(n)\star}(-u_n,+u_j^{\perp})$ satisfies Hall's $d$-harem condition.} \\ 
We start with an application of Lemma \ref{access}.
Consider the induced subgraph of the graph $\Gamma^{(n)\star}(-u_n,-\mathfrak{v})$ that consists of all vertices adjacent to the edges of the matching $\mathfrak{M}^1_n$. 
We take it as the graph $\Gamma^{\star}$ from Lemma \ref{access}. 
Since $\mathfrak{M}^1_n$ is a perfect $(1,d)$-matching in the ball $\Gamma^{(n)\star}\cap\mathcal{B}(u_n)$, it follows that $\Gamma^{\star}$ satisfies Hall's $d$-harem condition. 
The matching $\mathfrak{M}^1_n$ plays the role of $M$ from that lemma and the bipartite graph $\Gamma^{(n)}(-u_n)$ plays the role of $\Gamma$. 
Apply Lemma \ref{access} for $\mathfrak{v}$, $X$, $u^{\perp}_j$.   
We see $u^{\perp}_j \xhookdoubleheadrightarrow{\mathfrak{M}^1_n,X} \mathfrak{v}$. 
This gives us sequences of vertices $\{v_0',\ldots,v'_n\}, \{u_0',\ldots,u'_{n-1}\}$ as in Definition \ref{acc}.

In order to prove that the graph $\Gamma^{(n)\star}(-u_n,+u_j^{\perp})$ satisfies Hall's $d$-harem condition, we  construct a perfect $(1,d)$-matching in the ball $\Gamma^{(n)\star}(-u_n,+u_j^{\perp})\cap \mathcal{B}(u_n)$. 
We set
$$
M':=(\mathfrak{M}^1_n\setminus\{(u_n,v^0_{n,1}),\ldots, (u_n,v^0_{n,d}),(u'_{0},v'_{0}),\ldots (u'_{n-1},v'_{n-1})\}) \cup
$$ 
$$
\{(u^{\perp}_j,v^{\perp}_{j,1}),\ldots,(u^{\perp}_j,v^{\perp}_{j,d-1}),(u^{\perp}_j,v'_{0}), (u'_{0},v'_{1}),\ldots (u'_{n-1},v'_{n})\},
$$ 
where $v'_n=\mathfrak{v} \in \{ v^0_{n,1},\ldots, v^0_{n,d}\}$.

\begin{figure}[ht]
  \centering
    \includegraphics[width=0.7\textwidth]{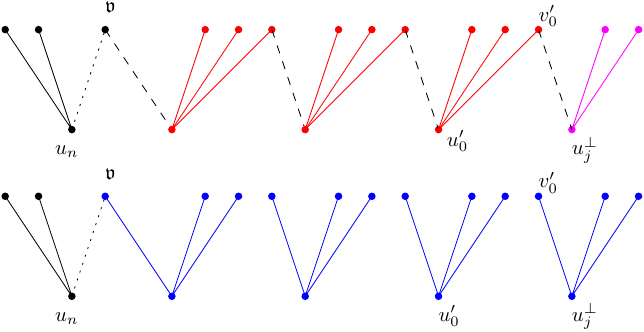}
	\caption{We replace the red fans in the matching $\mathfrak{M}^1_n$ by the blue fans to obtain the matching $M'$ in $\Gamma^{(n)\star}(-u_n,+u_j^{\perp})$.}
\end{figure}

%%%%%%%%%%
%We remind the reader that $\mathfrak{M}^1_n$ was a perfect $(1,d)$-matching in the ball $\Gamma^{(n)\star}\cap\mathcal{B}(u_n)$. 
%%%%%%%%%%%%%%%%%%%
We have obtained $M'$ by removing $d$ edges incident to $u_n$, adding $d$ edges incident to $u^{\perp}_j$, and the following replacement: for each of $u'_i$ we replace one edge incident to it by another incident edge (then $\mathfrak{v}$ becomes adjacent to one edge in $M'$). 
It follows that the matching $M'$ satisfies the conditions of the perfect $(1,d)$-matchings for all susbsets of $\Gamma^{(n)\star}(-u_n,+u_j^{\perp})\cap (\mathcal{B}(u_n)\setminus \mathcal{S}(u_n))$.

To finish the proof, note that the conditions of Lemma \ref{hccom} are satisfied if we consider $\Gamma^{(0)}$ as $\Gamma$ in this lemma, $\Gamma^{(n)\star}$ as $\Gamma^{\star}$,  $M'$ as $\mathfrak{M}$ and $\Gamma^{(n)\star}(-u_n,+u_j^{\perp})$ as $(\Gamma^{\star}\setminus \mathfrak{F}_u)\cup\mathfrak{S}$, i.e. $u_n$ is considered as $u$ and $u^{\perp}_j$ as $u^{\perp}$. 
Therefore, by Lemma \ref{hccom}, $\Gamma^{(n)\star}(-u_n,+u_j^{\perp})$ satisfies Hall's $d$-harem condition.
\end{proof}

\subsection{Notation used in proof of Lemma \ref{2nd}.}
Before stating the second lemma, we remind the reader the relevant notation from the construction.
\begin{itemize}
\item In Case 3A
\begin{itemize}
\item $\mathfrak{U}:=U^{(n+1)}\setminus U^{(n)\perp}$
\item $\mathfrak{V}:=V^{(n+1)}\setminus V^{(n)\perp}$
\item $\mathfrak{T}:=(\mathfrak{U},\mathfrak{V},\mathfrak{E})$, where $\mathfrak{E}$ is induced in $\Gamma$ by the sets of vertices $\mathfrak{U},\mathfrak{V}$. 
\item $\dot{U}^{(n)\perp}=U^{(n)\perp}\cap U^{(n+1)}$;\item $\dot{V}^{(n)\perp}=V^{(n)\perp}\cap V^{(n+1)}$
\item $\dot{\Gamma}^{(n)\perp}:=(\dot{U}^{(n)\perp},\dot{V}^{(n)\perp},\dot{E}^{(n)\perp})$, where $\dot{E}^{(n)\perp}$ is defined according to the set of fans of elements from $\dot{U}^{(n)\perp}$ putted into $\dot{\Gamma}^{(n)\perp}$.
\end{itemize}
\item In Case 3B
\begin{itemize}
\item $\mathfrak{U}:=U^{(n+1)}\setminus (U^{(n)\perp}\cup \{\dot{u}_n^{\perp}\})$ 
\item $\mathfrak{V}:=V^{(n+1)}\setminus (V^{(n)\perp}\cup \{\dot{v}_{n,i}^{\perp}: 1\leq i\leq d-1\})$.
\item $\mathfrak{T}:=(\mathfrak{U},\mathfrak{V},\mathfrak{E})$, where $\mathfrak{E}$ is induced in $\Gamma$ by the sets of vertices $\mathfrak{U},\mathfrak{V}$.\item $\dot{U}^{(n)\perp}:= (U^{(n)\perp}\cup\{\dot{u}_n^{\perp}\})\cap U^{(n+1)}$ with $\dot{V}^{(n)\perp}$ defined accordingly: $\dot{V}^{(n)\perp}:=(V^{(n)\perp}\cup \{\dot{v}_{n,k}^{\perp}:1\leq k \leq d-1\})\cap V^{(n+1)}$; 
\item $\dot{\Gamma}^{(n)\perp}:=(\dot{U}^{(n)\perp},\dot{V}^{(n)\perp},\dot{E}^{(n)\perp})$, where $\dot{E}^{(n)\perp}$ is defined according to the set of fans of elements from $\dot{U}^{(n)\perp}$ putted into $\dot{\Gamma}^{(n)\perp}$.
\end{itemize}
\item $\mathfrak{M}^2_n$ is a finite $(1,d)$-matching in the bipartite graph 
$\mathcal{B}^{(n)\star}(u^{l+1}_n)=\mathcal{B}^{(n)}(u^{l+1}_n)\cap \Gamma^{(n)\star}$. 
It satisfies the conditions of the perfect $(1,d)$-matchings for all vertices that are at the distance less than $\max\{4h(3d(n+1))+3,5\}$ from $u^{l+1}_n$; 
\item $\dot{v}^{l+1}_{n,i}$, $1 \le i \le d$ (or $d-1$), are vertices adjacent to $u^{l+1}_n$ in $\mathfrak{M}^2_n$ (or in $\Gamma^{(n)\perp}$);
\item $\dot{v}_1,\dot{v}_2$ are the elements among $\dot{v}^{l+1}_{n,i}$, $1 \le i \le d$ (or $d-1$), which are not added to $M^{l+1}_n$; there are at most two of them. 
\end{itemize}

\begin{lm}\label{2nd}
The bipartite graph $\mathfrak{T}$ is highly computable for any $n$. 
Moreover, one of the following holds:
\begin{itemize}
\item $\mathfrak{T}$ satisfies $c.e.H.h.c.(d)$;
\item the graph $\mathfrak{T}(+u_j^{\perp})$ satisfies $c.e.H.h.c.(d)$ for some $u_j^{\perp}\in \dot{U}^{(n)\perp}$;
\item the graph $\mathfrak{T}(+u_i^{\perp},+u_j^{\perp})$ satisfies $c.e.H.h.c.(d)$  for some $u_i^{\perp},u_j^{\perp}\in \dot{U}^{(n)\perp}$.
\end{itemize}
In the latter cases, the elements $u_i^{\perp},u_j^{\perp}$ can be computed and the corresponding bipartite graphs are highly computable. 
\end{lm}

\begin{rem}
If $|\dot{U}^{(n)\perp}|\leq 1$, then Lemma \ref{2nd} can be restated as follows. 
One of	the following holds:
\begin{itemize}
\item $\mathfrak{T}$ is a highly computable bipartite graph satisfying $c.e.H.h.c.(d)$;
\item $\Gamma^{(n+1)}$ is a highly computable bipartite graph satisfying $c.e.H.h.c.(d)$.
\end{itemize}
Therefore this Lemma proves Claim \ref{c1p2ef}.
\end{rem}

\begin{proof}[Proof of Lemma \ref{2nd}]
As in the proof of Lemma \ref{1st} we begin by showing that each of graphs from the formulation satisfies $c.e.H.h.c.(d)$ as long as it satisfies Hall's $d$-harem condition.

Let 
$$
\widehat{h}_1(m) = \left\{
\begin{array}{rr}
0, \quad \text{if} \quad m=0, \ \\
h(m+(3n+3)d), \quad \text{if} \quad m > 0.
\end{array}\right.,$$

$$
\widehat{h}_2(m) = \left\{
\begin{array}{rr}
0, \quad \text{if} \quad m=0, \ \\
h(m+(3n+2)d), \quad \text{if} \quad m > 0.
\end{array}\right.
$$
and
$$
\widehat{h}_3(m) = \left\{
\begin{array}{rr}
0, \quad \text{if} \quad m=0, \ \\
h(m+(3n+1)d), \quad \text{if} \quad m > 0.
\end{array}\right.
$$
\begin{clm}
Assume that one of the graphs $\mathfrak{T}$, or $\mathfrak{T}(+u_j^{\perp})$ or $\mathfrak{T}(+u_i^{\perp},+u_j^{\perp})$ satisfies Hall's $d$-harem conditions. 
Then in the first case the function $\widehat{h}_1$ witnesses $c.e.H.h.c.(d)$ for $\mathfrak{T}$, in the second case the function $\widehat{h}_2$ witnesses $c.e.H.h.c.(d)$ for $\mathfrak{T}(+u_j^{\perp})$ and in the third case the function $\widehat{h}_3$ witnesses $c.e.H.h.c.(d)$ for $\mathfrak{T}(+u_i^{\perp},+u_j^{\perp})$.
\end{clm}
\begin{proof}
It is a counterpart of Claim \ref{hall1}.  
Apply Lemma \ref{chc}. 
In the case of $\mathfrak{T}$, we view $\Gamma^{(0)}$ as $\Gamma$ and $\mathfrak{T}$ as $\Gamma^{\star}$ of that lemma. 
Furthermore, the number $\ell$ in its formulation is equal to the number of fans in both $\bigcup\limits_{i=1}^{n}M_{i-1}$ and $\Gamma^{(n)\perp}$ and then increased by 3  
(the number 3 corresponds to the fans of $u_n, u^{l+1}_n$ and $\dot{u}^{(n)\perp}$). 
Thus, $\ell \le 3n+3$.  
Since $\mathfrak{T}$ satisfies Hall's $d$-harem conditions, the lemma works.

In the second case we view $\Gamma^{(0)}$ as $\Gamma$ and $\mathfrak{T}(+u_j^{\perp})$ as $\Gamma^{\star}$ of Lemma \ref{chc}. 
The number $\ell$ is equal to the number of fans in both $\bigcup\limits_{i=1}^{n}M_{i-1}$ and $\Gamma^{(n)\perp}$ plus 3 (corresponding to the fans of $u_n, u^{l+1}_n$ and $\dot{u}^{(n)\perp}$) and minus 1 (corresponding to the fan of $u_j^{\perp}$). 
Therefore, it is at most $3n+2$.  
Since $\mathfrak{T}(+u_j^{\perp})$ satisfies Hall's $d$-harem conditions, Lemma \ref{chc} works again.

In the third case $\Gamma^{(0)}$ and $\mathfrak{T}(+u_i^{\perp},+u_j^{\perp})$ become $\Gamma$ and  $\Gamma^{\star}$. 
The number $\ell$ from Lemma \ref{chc} is equal to the number of fans in both $\bigcup\limits_{i=1}^{n}M_{i-1}$ and $\Gamma^{(n)\perp}$ plus 3 (corresponding to the fans of $u_n, u^{l+1}_n$ and $\dot{u}^{(n)\perp}$) and minus 2 (corresponding to the fans of $u_i^{\perp},u_j^{\perp}$). 
Therefore, $\ell \le 3n+1$. 
Since $\mathfrak{T}(+u_i^{\perp},+u_j^{\perp})$ satisfies Hall's $d$-harem conditions, Lemma \ref{chc} works again.
\end{proof}

Since (assuming that $u^{\perp}_i,u^{\perp}_j$ are computed) each of the graphs from the claim is a highly computable bipartite graph, it is enough to show that at least one of them satisfies Hall's $d$-harem condition (then apply the claim). 

Assume that $\mathfrak{T}$ does not satisfy Hall's $d$-harem condition. 
Let $u^{l+1}_n$ be the root of the last fan added to the matching $M_n$ in the second part of the $n$-th step; it belongs to the produced cycle of length $2$.
Recall that $\dot{v}_1,\dot{v}_2$ denote the vertices of the form $\dot{v}^{l+1}_{n,i}$ that were not added to $M^{l+1}_n$.
In particular in Case 3A there is only one such a vertex and in Case 3B there is either one or two such vertices. 

In the rest of the proof we will use the following two claims.

\begin{clm}\label{gn0}
Depending on the existence of $\dot{v}_2$ either the graph $\mathfrak{T}(-\dot{v}_1)$ or $\mathfrak{T}(-\dot{v}_1,-\dot{v}_2)$ satisfies Hall's $d$-harem condition.
\end{clm}

\begin{proof}
We only consider the case where $\dot{v}_2$ exists; the other case is analogous. 
It follows from the procedure of Section 5 that:
$$
\mathfrak{U}(-\dot{v}_1,-\dot{v}_2)=U^{(n)\star}(-u_n)\setminus\{u^{l+1}_n,\dot{u}^{\perp}_n\} 
$$
and 
$$
\mathfrak{V}(-\dot{v}_1,-\dot{v}_2)=V^{(n)\star}(-u_n)\setminus\{\dot{v}^{l+1}_{n,1},\ldots, \dot{v}^{l+1}_{n,d},\dot{v}^{\perp}_{n,1},\ldots, \dot{v}^{\perp}_{n,d}\}.
$$ 
Since $\mathfrak{M}^2_n$ is a perfect $(1,d)$-matching in the ball $\mathcal{B}^{(n)\star}(u^{l+1}_n)$ and $\mathfrak{T}(-\dot{v}_1,-\dot{v}_2)$ is obtained from $\Gamma^{(n)\star}(-u_n)$ by removing two fans from $\mathfrak{M}^2_n$, by Proposition \ref{hallincom} we know that $\mathfrak{T}(-\dot{v}_1,-\dot{v}_2)$ satisfies Hall's $d$-harem condition. 
\end{proof}

\begin{clm}\label{gn1}
For any $X\subset V^{(n+1)}$, we have
$$|N_{\Gamma^{(n+1)}}(X)|\geq\frac{1}{d}|X|.$$
\end{clm}
\begin{proof}
The inequality follows from Lemma \ref{neighsize}.
Indeed, consider $\Gamma^{(n+1)}$ to be $\Gamma$ from that lemma. 
Note that the construction of Section 5 guarantees that $\Gamma^{(n+1)}$ is $U^{(n+1)}$-reflected. 
Depending on existence of $\dot{v}_2$, put either $\dot{\Gamma}^{(n)\perp}(+\dot{v}_1)$ or $\dot{\Gamma}^{(n)\perp}(+\dot{v}_1,+\dot{v}_2)$ to be $\Gamma^{\perp}$ from that lemma. 
Then the corresponding graph $\Gamma^{\star}$ from the lemma is equal to either $\mathfrak{T}(-\dot{v}_1)$ or $\mathfrak{T}(-\dot{v}_1,-\dot{v}_2)$ and by Claim \ref{gn0} it satisfies Hall's $d$-harem condition. 
Therefore, the conditions of Lemma \ref{neighsize} are satisfied.
\end{proof}

\noindent 
We now return to the main course of the proof. From now on we consider the case of two additional vertices: $\dot{v}_1$ and $\dot{v}_2$. 
The case of only one of them, $\dot{v}_1$, is similar and simpler (in particular, we do not need two parts in the further proof). 

It is clear that the inequality $|N_{\mathfrak{T}}(X)|\geq |N_{\mathfrak{T}(-\dot{v}_1,-\dot{v}_2)}(X)|$ holds for all $X\subset \mathfrak{U}$.
This inequality also holds for subsets of $\mathfrak{V}$ which do not intersect $\{\dot{v}_1,\dot{v}_2\}$. 
On the other hand, applying Claim \ref{gn0} we see that the graph $\mathfrak{T}(-\dot{v}_1,-\dot{v}_2)$ satisfies Hall's $d$-harem condition. 
%%%%%
%$\bullet$ Let $\mathfrak{M}$ denote the perfect %$(1,d)$-matching in 
%$\mathfrak{T}(-\dot{v}_1,-\dot{v}_2)$. 
%%%%%%%%%%%%%%%%%%%
Therefore, if $\mathfrak{T}$ does not satisfy  Hall's $d$-harem condition, then it would be  witnessed by some finite subset of $\mathfrak{V}$ containing at least one of $\dot{v}_1,\dot{v}_2$.

The rest of the proof is divided into two parts. 

\noindent 
$\bullet$ \textit{Part 1.}
In this part of the proof we define a graph $\Gamma'$, a matching $M'$ and a family of fans $\dot{\Gamma}^{(n+1)\perp}$.  
We start with checking whether the graph $\mathfrak{T}(-\dot{v}_2)$ satisfies Hall's $d$-harem condition. 
If it does, we set 
\begin{itemize}
\item $M'$ to be a matching obtained by Proposition \ref{finmat} in $\mathfrak{T}(-\dot{v}_2)\cap \mathcal{B}(u^{l+1}_n)$;
\item $\dot{\Gamma}^{(n+1)\perp}:=\dot{\Gamma}^{(n)\perp}$;
\item $\Gamma' :=\mathfrak{T}$, denoting $\Gamma'=(U',V')$,  
\end{itemize} 
and finish Part 1 of the proof.

If the graph $\mathfrak{T}(-\dot{v}_2)$ does not satisfy Hall's $d$-harem condition, then there exists a minimal connected set $X$ such that $\dot{v}_1\in X\subset \mathfrak{V}(-\dot{v}_2)$ and $|N_{\mathfrak{T}(-\dot{v}_2)}(X)|<\frac{1}{d}|X|$.
Applying the argument of Claim \ref{ins}, one can show that such $X$ is a subset of $(\mathcal{B}(u^{l+1}_n)\setminus \mathcal{S}(u^{l+1}_n))$. 
In particular, such a set $X$ can be computed.

Observe that 
$|N_{\Gamma^{(n+1)}}(X)|\geq\frac{1}{d}|X|$
by Claim \ref{gn1}.

The inequalities 
$$
|N_{\Gamma^{(n+1)}}(X)|\geq\frac{1}{d}|X| \, \mbox{ and } 
\, |N_{\mathfrak{T}(-\dot{v}_2)}(X)|<\frac{1}{d}|X|
$$
imply 
$$ 
N_{\Gamma^{(n+1)}}(X)\cap \dot{U}^{(n)\perp}\neq\emptyset,
$$
i.e. there exists some $u^{\perp}_j\in N_{\Gamma^{(n+1)}}(X)$. 
Since $\Gamma^{(n+1)}$ is highly computable, the element $u^{\perp}_j$ and its fan from $\dot{\Gamma}^{(n)\perp}$ can be computed. Similarly as in Lemma \ref{1st} we denote by $v^{\perp}_{j,1},\ldots, v^{\perp}_{j,d-1}$ the remaining vertices of the fan containing $u^{\perp}_j$.

Consider the graph $\Gamma^{\star}$ induced in the graph $\mathfrak{T}(-\dot{v}_2)$ by the vertices adjacent to the edges of the matching $\mathfrak{M}^2_n$. 
Since this matching is a perfect $(1,d)$ matching in the ball $\mathcal{B}(u^{l+1}_n)\cap \Gamma^{(n)\star}$, it follows that $\Gamma^{\star}$ satisfies Hall's $d$-harem condition.
Under the circumstances of a bipartite graph $\Gamma^{(n+1)}$ and its subgraph $\Gamma^{\star}$, apply Lemma \ref{access} for $u^{\perp}_j$, $X$, $\dot{v}_1$ and arrive at
$u^{\perp}_j \xhookdoubleheadrightarrow{\mathfrak{M}^2_n,X} \dot{v}_1$. 
This gives us sequences $\{v_0',\ldots,v'_n\}, \{u_0',\ldots,u'_{n-1}\}$ as in Definition \ref{acc}.

We now apply an argument similar to one from the proof of Lemma \ref{1st}. 
We set
\begin{align*}
M':=(\mathfrak{M}^2_n\setminus\{(u^{l+1}_n,\dot{v}^{l+1}_{n,1}),\ldots, (u^{l+1}_n,\dot{v}^{l+1}_{n,d}),(u'_{0},v'_{0}),\ldots (u'_{n-1},v'_{n-1})\})\cup\\ \{(u^{\perp}_j,v^{\perp}_{j,1}),\ldots, (u^{\perp}_j,v^{\perp}_{j,d-1}), (u^{\perp}_j,v'_{0}), (u'_{0},v'_{1}),\ldots, (u'_{n-1},v'_{n})\}, 
\end{align*} 
where $v'_{n} =\dot{v}_1$. 
We have obtained $M'$ from $\mathfrak{M}^2_n$ by removing the $d$-fan of $u^{l+1}_n$, adding the $d$-fan of $u^{\perp}_j$, and the following replacement: for each of $u'_i$ we replace one edge incident to it by another incident edge (then $\dot{v}_1$ becomes adjacent to one edge in $M'$).
Observe that 
$$
(\mathfrak{U}\cup\{u_j^{\perp}\},(\mathfrak{V}\cup\{v_{j,1}^{\perp},\ldots, v_{j,d-1}^{\perp}\})\setminus \{\dot{v}_2\})=\mathfrak{T}(+u_j^{\perp})\setminus \{\dot{v}_2\}=\mathfrak{T}(+u_j^{\perp},-\dot{v}_2).
$$ 
Since $\mathfrak{M}^2_n$ was a $(1,d)$-perfect matching in the ball $\mathcal{B}(u^{l+1}_n)\cap \Gamma^{(n)\star}$,
the matching $M'$ satisfies the conditions of the perfect $(1,d)$-matchings for all susbsets of $\mathfrak{T}(+u_j^{\perp},-\dot{v}_2)\cap (\mathcal{B}(u^{l+1}_n)\setminus \mathcal{S}(u^{l+1}_n))$.
Therefore, $M'$ is a perfect $(1,d)$-matching in the ball $\mathfrak{T}(+u_j^{\perp},-\dot{v}_2)\cap \mathcal{B}(u^{l+1}_n)$.

The conditions of Lemma \ref{hccom} are satisfied if we consider $\Gamma^{(0)}$ as $\Gamma$ in this lemma, $\Gamma^{(n)\star}(-u_n)$ as $\Gamma^{\star}$, $M'$ as $\mathfrak{M}$, and the  graph $\mathfrak{T}(+u_j^{\perp},-\dot{v}_2)$ as $(\Gamma^{\star}\setminus \mathfrak{F}_u)\cup\mathfrak{S}$, i.e. $u^{l+1}_n$ is considered as $u$, and $u^{\perp}_j$ is considered as $u^{\perp}$ in that lemma. 
Therefore, by Lemma \ref{hccom} the graph $\mathfrak{T}(+u_j^{\perp},-\dot{v}_2)$ satisfies Hall's $d$-harem condition.

We define $\Gamma':=\mathfrak{T}(+u_j^{\perp})$, denote $\Gamma' = (U',V')$ and put 
$$\dot{\Gamma}^{(n+1)\perp}=(\dot{U}^{(n)\perp}\setminus \{u_j^{\perp}\},\dot{V}^{(n)\perp}\setminus \{v_{j,k}^{\perp}:1\leq k \leq d-1\}).$$
This ends the first part of the proof.

\noindent 
$\bullet$ \textit{Part 2.}
We check whether the graph $\Gamma'$ satisfies Hall's $d$-harem condition.
If it does, then by Part 1, $\Gamma'$ has to be equal to $\mathfrak{T}(+u_j^{\perp})$ and 
the proof is finished by the second option of the formulation.

If the graph $\Gamma'$ does not satisfy Hall's $d$-harem condition, then repeating the reasoning of Part 1, we see that there is (and it can be effectively found) a minimal connected set $X$ such that $\dot{v}_2\in X\subset V'\cap(\mathcal{B}(u^{l+1}_n)\setminus\mathcal{S}(u^{l+1}_n))$ and $|N_{\Gamma'}(X)|<\frac{1}{d}|X|$.  
Again, using Claim \ref{gn1} we obtain
$$
N_{\Gamma^{(n+1)}}(X)\cap \dot{U}^{(n)\perp}\neq\emptyset,  
$$ 
i.e. there exists some $u^{\perp}_i\in N_{\Gamma^{(n+1)}}(X)$. 
By high computability of $\Gamma^{(n+1)}$, this element can be computed. 
As usual, we denote by $v^{\perp}_{i,1},\ldots, v^{\perp}_{i,d-1}$ the vertices adjacent to $u^{\perp}_j$ in $\dot{\Gamma}^{(n+1)\perp}$.

The matching $M'$ obtained in the first part of the proof is a perfect $(1,d)$-matching either in the ball $\mathfrak{T}(-\dot{v}_2)\cap \mathcal{B}(u^{l+1}_n)$, or in the ball $\mathfrak{T}(+u_j^{\perp},-\dot{v}_2)\cap \mathcal{B}(u^{l+1}_n)$. 
Each of them is a subgraph of $\Gamma^{(n+1)}$. 
Applying Lemma \ref{access} in the appropriate case, for $u^{\perp}_i, X,\dot{v}_2$,   
we have 
$u^{\perp}_i \xhookdoubleheadrightarrow{M',X} \dot{v}_2$. 
Again, by the argument of Lemma \ref{1st} we obtain a matching 
$$
M'':=(M'\setminus\{(u'_{0},v'_{0}),\ldots (u'_{n-1},v'_{n-1})\})\cup$$
$$\{(u^{\perp}_i,v^{\perp}_{i,1}),\ldots, (u^{\perp}_i,v^{\perp}_{i,d-1}),(u^{\perp}_i,v'_{0}), (u'_{0},v'_{1}), \ldots,(u'_{n-1},v'_{n})\}.
$$
We have obtained $M''$ from $M'$ by adding $d$ edges incident to $u^{\perp}_i$, and replacing one edge in the matching for each of $u'_k$ in such a way, that $\dot{v}_2$ is adjacent to an edge in $M''$. 
As a result, $M''$ is the perfect $(1,d)$-matching in the ball $\Gamma'(+u_i^{\perp})\cap (\mathcal{B}(u^{l+1}_n)$.

Now the conditions of Lemma \ref{hccom} are satisfied if we consider $\Gamma^{(0)}$ as $\Gamma$ in this lemma, $\Gamma^{(n)\star}(-u_n)$ as $\Gamma^{\star}$, $M''$ as $\mathfrak{M}$, and $\Gamma'(+u_i^{\perp})$ as $(\Gamma^{\star}\setminus \mathfrak{F}_u)\cup\mathfrak{S}$, i.e. $u^{l+1}_n$ is considered as $u$, and $u^{\perp}_i$ as $u^{\perp}$ in that lemma. 
Therefore, by Lemma \ref{hccom}, $\Gamma'(+u_i^{\perp})$ satisfies Hall's $d$-harem condition.

The final argument depends on two possible outputs of Part 1.  
If the graph $\mathfrak{T}(-\dot{v}_2)$ does not satisfy  Hall's $d$-harem condition (i.e. $u^{\perp}_j$ is involved), then $M''$ is a perfect $(1,d)$-matching in the graph $\mathfrak{T}(+u_i^{\perp},+u_j^{\perp})$. 
In the contrary case we redefine $u^{\perp}_j:=u^{\perp}_i$ and then $M''$ becomes a perfect $(1,d)$-matching in the graph $\mathfrak{T}(+u_j^{\perp})$. 
Therefore, if $\mathfrak{T}$ does not satisfy Hall's $d$-harem condition, then either $\mathfrak{T}(+u_j^{\perp})$ or $\mathfrak{T}(+u_i^{\perp},+u_j^{\perp})$ satisfies this condition.
\end{proof}

\section{Proof of the main theorem}\label{pMT}
We remind the formulation of Theorem \ref{ehhc}. 

\bigskip 

\noindent 
{\em 
Let $\Gamma=(U,V,E)$ be a highly computable bipartite graph, such that: 
\begin{itemize}
\item both $U$ and $V$ are identified with $\mathbb{N}\setminus\{0\}$
\item $E$ does not contain edges of the form $(u,v_u)$
\item  $\Gamma$ is fully reflected,  
\item $\Gamma$ satisfies the $c.e.H.h.c.(d)$.
\end{itemize}
Then there exist a computable perfect $(1,d-1)$-matching of $\Gamma$, which realizes a computable $(d-1)$ to $1$ function $f:\mathbb{N}\rightarrow \mathbb{N}$ with controlled sizes of its cycles. } 

\bigskip 

\begin{proof}
Let us apply the construction of Section \ref{const}.
This construction works modulo Claims \ref{c1p1ef}, \ref{c1p2ef}, \ref{csnp1ef}, \ref{cs1p2ef}. 
Claims \ref{c1p1ef} and \ref{csnp1ef} follow from Lemma \ref{1st}. Claims \ref{c1p2ef} and \ref{cs1p2ef} follow from Lemma \ref{2nd}.
Since for every $n$ the union  $\bigcup\limits_{i=0}^n M_i$ consists of $(1,d-1)$-fans, the final set of edges $M$ is a $(1,d-1)$-matching.

For every $u\in U$ there is a step where an edge incident to $u$ is added to $M$. 
Furthermore, if the copy $v_u$ was not added to $M$ earlier, the construction ensures that $v_u$ is also added to $M$ in the second part of this step. 
This guarantees that $M$ is a perfect $(1,d-1)$-matching of the graph $\Gamma$. 
It also implies computability of this matching.  
Indeed, it is easy to see from the description given in Section 5 that a vertex $n$ 
appears in $M$ at the $n$-th step at latest. 
This follows from minimality of $u_n$ in $U^{(n-1)}$.  
Thus, for any $n$ all edges from $M$ of the forms $(n,j)$ and $(k,n)$ (the unique one) are already in $\bigcup\limits_{i=0}^n M_i$. 
Therefore, there is an algorithm which computes all such pairs from $M$ for any input $n$. 
We see that the function $f$ realized by this matching is computable. 

Until the end of the proof we will abuse the notation and use a natural number $n$ both for vertices and numbers of steps.

It remains to show that $f$ has controlled sizes of its cycle. 
To see condition $(i)$ of Definition \ref{cycles} note that the edges $(u_0, v^0_{0,1})$ and $(v_{u_0},u_{v^0_{0,1}})$ are added to $M$ at step 1, i.e. $f^2(u_0)=f(u_{v^0_{0,1}})=u_0$. 
Since $u_0=1$, condition $(i)$ is satisfied.

As a vertex $n$ appears in $M$ at the $n$-th step at latest, the length of a cycle created at the $n$-th step cannot be greater than $\max\{2,n\}$.  
In particular, if $n$ is in a cycle then $f^i(n)=n$ for some $i\leq n$ and condition $(ii)$ of Definition \ref{cycles} is satisfied. 

It remains to show that condition $(iii)$ is satisfied.
Let $\mathsf{f}_i$ be the partial function (living in $\mathbb{N}$) realized by $M_i$. 
Then $M_i = \{(\mathsf{f}_i(m),m):m\in\mathsf{Dom}(\mathsf{f}_i)\}$ is the graph of $\mathsf{f}_i$.  
We denote by $f_n$ the partial function  realized by $\bigcup\limits_{i=0}^{n}M_i$.
Let $F_n = \{(f_n(m),m):m\in\mathsf{Dom} (f_n )\}$ be the graph of $f_n$ on $\mathsf{Dom} (f_n )\cup \mathsf{Rng} (f_n )$. 
Then each $M_i$ with $i\le n$, is a subgraph of $F_n$. 

Since each subgraph $M_i$ is connected,  $F_n$ has at most $n+1$ connected components. 
Furthermore, one of the following properties holds:
\begin{itemize}
\item $M_n$ is a connected component of $F_n$ and contains a cycle;
\item $M_n$ is connected and there is a vertex of degree $1$ in $F_{n-1}$ which is an image of a vertex from $M_n$.  
\end{itemize} 
The first possibility arises in cases (2) and (3) of part 2 of the $(n+1)$-th step. 
The second possibility is the result of case (1) of this part of the step (see Section \ref{nplus1p2} and Remarks before part 2). 
Since each element is included into $M$ together with a fan, this property guarantees that $F_n$ consists only of vertices of degree $1$ and $d$.  
When a vertex has degree $1$, then its $F_n$-neighborhood is exactly its $f_n$-image.
When a vertex has degree $d$, then its $f_n$-image and $d-1$ preimages are in $F_n$.
In particular, each connected component of $F_n$ contains a cycle. 
The length of the cycle is not greater than $\max\{n,2\}$.

Since the value $f_{n-1} (n)$ is defined, $n$ belongs to some connected component of $f_{n-1}$. 
Thus, there exist $k$ and $l$ such that $f_{n-1}^{k+l}(n)=f_{n-1}^k(n)$.
These $k$ and $l$ work for the equality $f^{k+l}(n)=f^k(n)$. 
It remains to show that $k$ can be bounded by $2n$ and $l$ by $n$. 

The latter estimate is easy: the biggest cycle that can be constructed before the $n$-th step, is not longer than $\max\{n-1,2\}$. 
Below we will only use the inequality $l \le n$ for simplicity. 

In order to show that $k$ is bounded by $2n$, let us estimate the size of a subset of $U$ that can be added to the matching $M$ in the process of the $n$-th step of the construction, in the situation when a new cycle is not created. 
It must consist of elements of $U^{(n-1)\perp}$ added to the matching at the iteration of part $2$ of the $n$-th step together with $u_n$.
Therefore, we can bound it by the maximal possible number $|U^{(n-1)\perp}|+1$.

Let us denote the number of elements from $U^{(s-1)\perp}$ added to $M$ at the $s$-th step by $\ell_s$.
Then for each $m\in M_s$ we have $f^{\ell_s +1}(m)\in M_j$ where $j \leq s-1$. 
If no cycle is constructed at the $j$-th step then $f^{\ell_{j}+1}(f^{\ell_{s}+1}(m))\in M_{i}$ for some $i\leq j-1$. 
Iterating this argument and applying it to $m:=n$ we arrive at 
$$
k\leq \sum\limits_{s=1}^n (\ell_s+1).
$$ 
Since at the $n$-th step the value $|\bigcup\limits_{s=1}^{n-1} U^{(s)\perp}|$ does not exceed $n-1$, 
$$
\sum\limits_{s=1}^n \ell_s \le n-1.
$$
We see that $k\leq 2n-1$. 
Thus condition $(iii)$ of Definition \ref{cycles} is satisfied.
\end{proof}

\section{Schneider's theorem and computable entourages of coarse spaces}\label{coarse}

We are ready to prove Theorem A from the introduction. 
We repeat it. 

\begin{thm}
Let $d\geq 3$. 
Let $(\mathbb{N},\mathcal{E})$ be a non-amenable coarse space of a bounded geometry. 
Assume that there exists a highly computable symmetric $E\in\mathcal{E}$ such that for every finite $F\subseteq \mathbb{N}$ we have $|E[F]|\geq (d+2)|F|$. 

Then there exists a computable $E'\in\mathcal{E}$ such that $\Gamma(E')$ is a $d$-regular forest. 
Moreover, there exists an algorithm which for every $m,n\in\mathbb{N}$ recognizes if $m$ and $n$ are in the same connected component of $\Gamma(E')$.
\end{thm} 

\begin{proof} 
Let $E$ be a highly computable symmetric entourage as in the formulation. 
Consider the graph $\Gamma(R)$ defined for the symmetric relation $R:= E\setminus \Delta_{\mathbb{N}}\subseteq \mathbb{N}\times\mathbb{N}$.
Since the coarse space $(\mathbb{N},\mathcal{E})$ is of a bounded geometry, the neighborhood of each vertex of $\Gamma (R)$ is finite.  
Further, the graph $\Gamma(R)$ is highly computable by high computability of $E$. 

Firstly, we apply Theorem \ref{ehhc} to $\Gamma (R)$. 
In order to do this, we only need to show that $\Gamma(R)$ satisfies $c.e.H.h.c.(d)$.
To see the latter, note:
$$
|R[F]|\geq |E[F]\setminus F|\geq |E[F]|-|F|\geq (d+1)|F|.
$$
Thus, for all finite sets $X\subset \mathbb{N}$:  
$$|N_{\Gamma(R)}(X)|-d|X|\ge (d+1)|X|-d|X|=|X|.
$$ 
In particular, the inequality $n\leq |X|$ implies $n\leq |N(X)|-d|X|\leq |N(X)|-\frac{1}{d}|X|$.
Since the identity map on $\mathbb{N}$ is a total computable function, it follows that $\Gamma(R)$ satisfies $c.e.H.h.c.(d)$.

Now Theorem \ref{ehhc} provides a computable function $f$ that realizes a perfect $(1,d-1)$ matching in the graph $\Gamma(R)$ and has  controlled sizes of its cycles. 
This function will be used in the construction of $E'\in\mathcal{E}$ as in the formulation of the theorem.  
We will adapt the proof of Theorem 2.2 of \cite{schndr}.

Since $f$ is a total, computable, surjective, and $(d-1)$ to $1$ function, the graph of $f$ (denoted by $\Gamma (f)$) is computable and $d$-regular. 
We remind the reader that $f$ satisfies the following properties:
\begin{enumerate}[label={(\roman*)}] 
\item $f^2(1)=1$;
\item if $n\geq 2$ and $f^i(n)=n$ then $i\leq n $;
\item if $n\geq 2$ and for all $i\leq n$ we have $f^i(n)\neq n$ then there exist $k\leq 2n$ and $l\leq n$ such that $f^{k+l}(n)=f^k(n)$; 
\item for each $n$ the pair $(n,f(n))$ belongs to $R$.
\end{enumerate}
Since $R=E\setminus \Delta_X$ the last property implies that $f$ does not have fixed points.
Let 
$$
P(f)=\{n\in \mathbb{N} \, | \, \exists m\geq 1 (f^m(n)=n) \} \, \mbox{ and } \, P_0 (f) = \{n\in P(f) \, | \, \forall m\geq 1 (f^m(n)\geq n) \}, 
$$ 
i.e. $P(f)$ is the union of all cycles and $P_0 (f)$ is the set of minimal representatives of cycles.  
By property (ii) there is an algorithm which for every $n\in\mathbb{N}$ verifies whether $n\in P(f)$, i.e. $P(f)$ is computable. 
Observe that $P_0 (f)$ is computable too. 
Indeed, the number $1$ obviously belongs to $P_0 (f)$. 
When $n\geq 2$ and $n\in P(f)$, then verifying if $f^i(n)\ge n$ for all $i\leq n$ we can check whether $n\in P_0 (f)$ (apply (ii) again). 

Since there is no component with two disjoint cycles, we see that whenever $n,m\in P_0(f)$ and $n\neq m$, then $n$ and $m$ do not belong to the same connected component of the graph $\Gamma (f)$.
Thus,
$$
P(f)=\dot{\bigsqcup}_{ n\in P_0 (f) } \{f^m(n) \, |  \, m\in\mathbb{N} \} .
$$ 
Further, there is an algorithm which for every $n\in \mathbb{N}$ finds  the 
$P_0(f)$-representative of the connected component of $n$. 
Indeed, if for example $n\notin P(f)$, then applying (iii) we compute $i\leq 2n$ such that $f^i(n)\in P(f)$, and later we find $j\leq n$ such that $f^{i+j}(n)\in P_0 (f)$.

Based on this we want to construct a new computable function $f_{\star}$ such that its graph (denoted by $\Gamma (f_{\star})$) is a computable $d$-regular forest. 
Let us start with two auxiliary functions $g,h: P_0 (f)\times \mathbb{N}\rightarrow \mathbb{N}$ such that, for all $n\in P_0 (f)$ and $m\geq 1$ the following properties hold: 
\begin{itemize}
\item $g(n,0)=n$ and $h(n,0)=f(n)$;
\item $\{g(n,m),h(n,m)\}\cap P(f)=\emptyset$;
\item $f(g(n,m))=g(n,m-1)$, and $f(h(n,m))=h(n,m-1)$.
\end{itemize} 
We want $g,h$ to be computable functions. 
Since $P_0 (f)$ and $P(f)$ are computable and the graph $\Gamma (f)$ is computable and $d$-regular, the following rule gives the required algorithm. 
Given $n\in P_0 (f)$ and $m\geq 1$ and having already defined $g(n,m-1)$, find the minimal $x$ such that $x \notin P(f)$ and $f(x)=g(n,m-1)$. 
Then let $g(n,m) =x$. 
The definition of $h(n,m)$ is similar. 
Clearly $g$ and $h$ are injective and have disjoint ranges.

Now we define $f_{\star}:\mathbb{N}\rightarrow\mathbb{N}$ for $x\in\mathbb{N}$ in the following way:
$$
f_{\star}(x)=\left\{\begin{array}{ll}
g(y,m+2) & \text{ if } x=g(y,m) \text{ for } y\in P_0 (f) \text{ and even } m\geq 0,\\
g(y,m-2) & \text{ if } x=g(y,m) \text{ for } y\in P_0 (f) \text{ and odd } m\geq 3,\\
f^2(x) & \text{ if } x=h(y,m) \text{ for } y\in P_0 (f) \text{ and } m\geq 2,\\
f(x) & \text{ otherwise.}
\end{array}\right.
$$ 
We remind the reader that for each $n$ there is $i\leq 3n$ such that $f^{i}(n) \in P_0 (f)$. 
Thus for any $m'>3n$ the number $n$ does not belong to $\{ g(y,m'), h(y,m')\}$. 
This guarantees that $f_\star$ is computable. 
The facts that $\Gamma (f_{\star})\in\mathcal{E}$, the function $f_{\star}$ does not have cycles, and for each $x\in\mathbb{N}$ the size $|f^{-1}_{\star}(x)|$ is $d-1$ are proved in \cite[proof of Theorem 2.2]{schndr}. 
Therefore, the graph $\Gamma (f_{\star})$ is a computable $d$-regular forest. 

To see the last statement of the theorem note that if $C\subset\mathbb{N}$ is a connected components of $\Gamma (f)$ then it is also the set of vertices of a connected component of $\Gamma (f_{\star})$. 
Since for each $n$ one can compute $n'\in P_0 (f)$ such that $n$ and $n'$ are in same connected component of $\Gamma (f)$, there is a straightforward algorithm which for any $n$ and $m$ verifies whether $n$ and $m$ are in the same tree in $\Gamma (f_{\star})$.
\end{proof}   

\bigskip 

We now translate Theorem A into a statement about wobbling groups. 
The following theorem is Theorem B from the introduction. 
It is a computable version of Corollary 2.3 of \cite{schndr}.

\begin{thm}
Let $(\mathbb{N},\mathcal{E})$ be a non-amenable coarse space of a bounded geometry. 
Assume that there exists a highly computable symmetric $E\in\mathcal{E}$ such that $|E[F]|\geq 6|F|$ for all finite $F\subset \mathbb{N}$.
Then there are two computable permutations $\sigma,\pi\in Sym(\mathbb{N})$ such that $\langle\sigma,\pi\rangle$ is a free semi-regular subgroup of $\mathcal{W}(\mathbb{N},\mathcal{E})$.
\end{thm}

\begin{proof}
By Theorem \ref{fg} there is $E'\in \mathcal{E}$ such that $\Gamma(E')$ is a computable $4$-regular forest. 
Moreover, there is an algorithm which for each $n,m\in\mathbb{N}$ recognizes if $n$ and $m$ are in the same connected component of $\Gamma(E')$.
Since $\Gamma (E')$ is a computable $4$-regular forest, there is another algorithm which for each $n$ finds  four natural numbers $n_1 < n_2< n_3< n_4$ such that for every $n_i$ either $(n,n_i)\in E'$ or $(n_i,n)\in E'$. 
Furthermore, for each input $n,m\in \mathbb{N}$, the output $B_m(n)\subset \mathbb{N}$, the $m$-ball of $n$, can be effectively found by a uniform algorithm. 

The standard Cayley graph of the free two generated group is isomorphic to the $4$-regular tree, so it is clear that there exist permutations $\sigma$,$\pi$ as in the formulation. We will show that there exist computable ones.
The construction is by induction. 
At every step we construct  two finite partial permutations of $\Gamma(E')$. 

At the first step of the construction of $\sigma$ and $\pi$ put $D_{1}=B_1(1)=\{1,1_1,1_2,1_3,1_4\}$ and define $\sigma_1$ and $\pi_1$ on $D_{1}$ so that $\sigma_1(1)=1_1,\sigma_1(1_2)=1, \pi_1(1)=1_3,\pi_1(1_4)=1$, and for every $i\le 4$, 
$\{ \sigma_1(1_i)^{\pm 1}, \pi_1 (1_i )^{\pm 1} \} \subset B_1(1_i )$ with $\sigma^{\pm 1}_1(1_i) \not=\pi^{\pm 1}_1 (1_i )$ (whenever the corresponding preimages are defined); the set $D_1 \cup \sigma_1 (D_1 ) \cup \pi_1 (D_1 )$ consists of 11 elements. 

Assume that before the $n$-th step of the construction two finite partial permutations $\sigma_{n-1}$ and $\pi_{n-1}$ 
are defined on some set $D_{n-1}$ such that  
for all $a\in D_{n-1}$, $\{ \sigma^{\pm 1}_{n-1} (a) , \pi^{\pm 1}_{n-1} (a) \} \subset B_1 (a)$.  

At step $n$ we first put 
$\sigma_n(a)=\sigma_{n-1}(a)$ and $\pi_n(a)=\pi_{n-1}(a)$ for all $a\in D_{n-1}$. 
Let $a_n$ be the first number in $\mathbb{N}\setminus D_{n-1}$. 
Let $X$ be the connected component of $\Gamma(E')$ containing $a_n$. 
If it does not meet $D_{n-1}$, then we find the corresponding $(a_n)_1,(a_n)_2,(a_n)_3,(a_n)_4$, and define $\sigma_n$ and $\pi_n$ on $\{ a_n ,(a_n)_1,(a_n)_2,(a_n)_3,(a_n)_4\}$ exactly as at the first step. 
Then put $D_n=D_{n-1}\cup B_1(a_n)$. 

It is worth mentioning at this stage that according to the assumptions on $\Gamma(E')$, the question if 
$X\cap D_{n-1} \not=\emptyset$ is decidable: just check if $a_n$ is in the same component with elements of $D_{n-1}$.  

Now, assume that the answer is positive. 
Then find the minimal $i<n$ such that $a_i\in X$ and the minimal $m$ such that $a_n\in B_m(a_i)$. 
Put $D_n=D_{n-1}\cup B_{m}(a_i)$. 
Extend $\sigma_{n-1} \upharpoonright_X$ and $\pi_{n-1} \upharpoonright_X$ to partial permutations defined on $B_{m}(a_i)$ such that for every $a\in B_m(a_i )$, $\{ \sigma^{\pm 1}_n (a ) ,\pi^{\pm 1}_n(a) \} \subset B_1(a)$, and $\sigma^{\pm 1}_n(a)\neq\pi^{\pm 1}_n(a)$ (whenever the corresponding preimages are defined). 

The set $B_m (a)$ can be viewed as the set of all $\mathsf{w}(a)$ where $\mathsf{w}$ is a group word 
of length $\le n$ of the alphabet $\{ \ell , r \}$. 
Then $\sigma_n$ (resp. $\pi_n$) takes $\mathsf{w} (a)$ with the image undefined before, to $\ell \mathsf{w} (a)$ (resp. $r \mathsf{w} (a))$. 

Let $\sigma=\lim\limits_{n\in\mathbb{N}} \sigma_n$, and $\pi=\lim\limits_{n\in\mathbb{N}} \pi_n$. 
These permutations act on each connected component of $\Gamma(E')$ by mapping elements to their neighbors.  
Clearly, there are no points which are fixed by non-identity elements of the group $\langle\sigma,\pi\rangle$. 
It is a free semi-regular subgroup of $\mathcal{W}(\mathbb{N},\mathcal{E})$. 
Permutations $\sigma$ and $\pi$ are computable by the construction. 
\end{proof}

\printbibliography

\end{document}